\documentclass[12pt,a4]{amsart}
\usepackage[a4paper, left=28mm, right=28mm, top=26mm, bottom=31mm]{geometry}
\usepackage{amsthm,amssymb}
\usepackage[all]{xy}
\usepackage{amsmath}
\usepackage{textcomp}
\usepackage{cases}
\usepackage{here}
\usepackage[dvipdfmx]{graphicx}
\usepackage{amscd}
\usepackage{mathrsfs}
\usepackage{float}
\usepackage{wrapfig}
\usepackage{color}
\usepackage[capitalize]{cleveref}

\newtheorem{thm}{Theorem}[section]
\newtheorem{dfn}[thm]{Definition}
\newtheorem{lem}[thm]{Lemma}
\newtheorem{prp}[thm]{Proposition}
\newtheorem{rem}[thm]{Remark}

\newtheorem{cor}[thm]{Corollary}

\newtheorem{asm}[thm]{Assumption}

\newtheorem{exam}[thm]{Example}
\def\P{\mathbb{P}}
\def\A{\mathscr{A}}

\def\R{\mathbb{R}}
\def\Z{\mathbb{Z}}
\def\Q{\mathbb{Q}}
\def\C{\mathbb{C}}
\def\t{\mathbf{t}}
\def\n{\mathbf{n}}
\def\D{\mathcal{D}}

\def\J{\mathcal{J}}

\def\E{\mathcal{E}}
\def\CC{\mathcal{C}}
\def\N{\mathcal{N}}

\def\H{\mathcal{H}}

\def\Re{\mathop{\mathrm{Re}}\nolimits}
\def\Prim{\mathop{\mathrm{Prim}}\nolimits}
\def\sign{\mathop{\mathrm{sign}}\nolimits}

\def\ess{\mathop{\mathrm{ess}}\nolimits}

\def\pqs{\mathop{\mathrm{pqs}}\nolimits}

\def\Tan{\mathop{\mathrm{Tan}}\nolimits}
\def\qs{\mathop{\mathrm{qs}}\nolimits}

\def\Tan{\mathop{\mathrm{Tan}}\nolimits}

\def\Dol{\mathop{\mathrm{Dol}}\nolimits}

\def\Id{\mathop{\mathrm{Id}}\nolimits}
\def\CH{\mathop{\mathrm{CH}}\nolimits}

\def\Ker{\mathop{\mathrm{Ker}}\nolimits}

\def\sing{\mathop{\mathrm{sing}}\nolimits}

\def\Trop{\mathop{\mathrm{trop}}\nolimits}

\def\relint{\mathop{\mathrm{rel.int}}\nolimits}

\def\Star{\mathop{\mathrm{Star}}\nolimits}

\def\supp{\mathop{\mathrm{supp}}\nolimits}

\def\Spec{\mathop{\mathrm{Spec}}\nolimits}

\def\Hom{\mathop{\mathrm{Hom}}\nolimits}

\def\Trop{\mathop{\mathrm{Trop}}\nolimits}

\numberwithin{equation}{section}
\numberwithin{figure}{section}

\makeatletter
\DeclareRobustCommand{\genericinterval}[2]{%
 \@ifstar{\genericinterval@star{#1}{#2}}{\genericinterval@nostar{#1}{#2}}}
  \newcommand{\genericinterval@star}[4]{\mathopen{}\mathclose{\left#1#3,#4\right#2}}
   \newcommand{\genericinterval@nostar}[4]{\mathopen{#1}#3,#4\mathclose{#2}}

        \makeatother

\begin{document}
\title[On tropical Hodge theory]{On tropical Hodge theory for tropical varieties}
\author{Ryota Mikami}
\address{Institute of Mathematics, Academia Sinica, Astronomy-Mathematics Building, No.\ 1, Sec.\ 4, Roosevelt Road, Taipei 10617, Taiwan.}
\email{mikami@gate.sinica.edu.tw}
\date{\today}
\maketitle

\begin{abstract}
  To prove log-concavity of the characteristic polynomials of matroids, 
  Adiprasito-Huh-Katz proved the K\"{a}hler package (the hard Lefschetz theorem and the Hodge-Riemann bilinear relations) for their Chow rings.
  Amini-Piquerez generalized it to tropical cohomology of smooth projective tropical varieties.
  Their proofs were combinatorial.
  In this paper, 
  we establish tropical Hodge theory except for regularity of solutions of Laplacian,
  and the K\"{a}hler package follows from tropical Hodge theory in the same way as compact K\"{a}hler manifolds. 
\end{abstract}

\setcounter{tocdepth}{1} 
\tableofcontents

\section{Introduction}
  \emph{Log-concavity} (i.e., $w_{k-1} ( M ) w_{k +1} ( M ) \leq w_k (M)^2$) of the absolute values $w_k (M)$ of the coefficients of the characteristic polynomial $\chi_M ( \lambda ) $ of a general matroid $M$ of rank $r +1$
  was conjectured by Welsh \cite{Wel76} in 1976.
  Adiprasito-Huh-Katz \cite{AHK18} proved it by showing \emph{the K\"{a}hler package} for its Chow rings $A^* ( M )$, 
  i.e., for $p \leq \frac{r}{2}$  and an ample element $l \in A^1 ( M ) \otimes \R$, 
    \begin{itemize}
      \item (the hard Lefschetz theorem \cite[Theorem 1.4 (1)]{AHK18}) 
            $$ A^p ( M ) \otimes \R \to A^{r  -p} ( M ) \otimes \R ,   \quad a \mapsto l^{r - 2p} \cdot a  $$
            is an isomorphism, 
      and 
      \item (the Hodge-Riemann bilinear relations \cite[Theorem 1.4 (2)]{AHK18}) a symmetric bilinear map 
            $$ A^p ( M ) \otimes \R \times A^p ( M ) \otimes \R \to \R , \quad  (a,b) \mapsto (-1)^p \deg(l^{r-2p} \cdot a  \cdot b )   $$
            is positive definite on the kernel of $l^{r-2p +1}$.
    \end{itemize}
(For more information on log-concavity, e.g., the case of realizable matroids, see \cite[Section 1]{AHK18}.)
Their proof was based on \emph{tropical geometry}.

  Amini-Piquerez \cite{AP20} generalized the K\"{a}hler package 
  to \emph{tropical singular cohomology} $H_{\Trop,\sing}^{p,q} ( Y )$ of smooth (i.e., locally matroidal) projective \emph{tropical variety} $Y$.
  (For a natural compactification $\overline{\Sigma_M}$ of the Bergman fan $\Sigma_M$ of a matroid $M$, 
  we have 
    $H_{\Trop,\sing}^{p,p} ( \overline{\Sigma_M}) \cong A^p ( M ) $ 
    and $H_{\Trop,\sing}^{p,q} ( \overline{\Sigma_M}) =0 $ $(p \neq q)$.)

Tropical singular cohomology $H_{\Trop,\sing}^{p,q} ( Y ) \otimes \R$
  is isomorphic to \emph{tropical Dolbeault cohomology} $ H_{\Trop,\Dol}^{p,q} (Y) $ (\cite[Theorem 1]{JSS19}), which is based on 
tropical analogs of differential $(p,q)$-forms on complex manifolds, called \emph{superforms}.
(Superforms were
  first introduced by Lagerberg (\cite{Lag12}) for $\R^n$, 
  and generalized by Chambert-Loir and Ducros (\cite{CLD12}) to tropical varieties and non-archimedean analytic spaces. 
  Many mathematicians have studied them, e.g., \cite{BGGJK21}, \cite{Gub16}, \cite{GJR21-1}, \cite{GJR21-2}, \cite{GK17}, \cite{Jel16A}, \cite{Jel16T},\cite{JSS19}, \cite{Mih21}.)

Although the K\"{a}hler package for singular (or Dolbeault) cohomology groups of compact K\"{a}hler manifolds 
  follows from \emph{Hodge theory} 
  (in particular, isomorphisms of Dolbeault cohomology groups and the spaces of harmonic forms), 
the known proofs (\cite{AHK18}, \cite{BHMPW22}, \cite{AP20}) of the above K\"{a}hler packages did not use differential geometry of tropical varieties.

In this paper, we study a tropical analog of Hodge theory.
The goal is the following.

\begin{thm}
  Let $\Lambda$ be a tropically compact projective tropical variety
  which is 
  $\Q$-smooth in codimension $1$ (\cref{definition of locally irreducible tropical varieties}).
  We assume that
  \cref{assumption regularity of Derichlet potentials} (regularity of solutions of tropical Laplacian) holds.

  Then the K\"{a}hler package holds for  a new tropical Dolbeault cohomology $H_{\Trop,\Dol,\pqs,\t''_{\min}}^{p,q} (\Lambda)$.  
\end{thm} 
\begin{rem}
When $\Lambda$ is smooth (in particular, $\Q$-smooth in codimension $1$), 
our new tropical Dolbeault cohomology $H_{\Trop,\Dol,\pqs,\t''_{\min}}^{p,q} (\Lambda)$ is isomorphic to the usual one $ H_{\Trop,\Dol}^{p,q} (\Lambda) $. 
\end{rem}

\begin{rem}
 Classically, regularity of solutions of Laplacian 
 on compact Riemann manifolds (without boundaries)
 are known (see \cite[p.380]{GH78} or \cite[Theorem 4.9]{Wel08}).
 That on compact Riemann manifolds with smooth boundaries are 
 also known (\cite[Theorem 2.2.6]{Sch95}), 
 see also \cref{remark on regularity of solutions of Lalpacians elliptic partial differential operators}.

  However, there is a  tropically compact projective tropical variety
  which is $\Q$-smooth in codimension $1$ 
  and does not satisfy \cref{assumption regularity of Derichlet potentials} (\cite[Section 5]{BH17}). See \cref{Counter example of Assumption} for details.
\end{rem}

Our tropical Hodge theory imitates those of compact K\"{a}hler manifolds and of compact Riemann manifolds with smooth boundaries.
We shall introduce 
\begin{itemize}
  \item  generalizations of Largerberg's $\R$-K\"{a}hler forms to tropical varieties, 
  \item  boundary conditions $\t''_{\min}, \t''_{\max}, \n''_{\min}$ and $ \n''_{max}$ of differential forms, and 
  \item    $\overline{*}$-operator, codifferential $\overline{\delta} := - \overline{*} \overline{\partial} \overline{*}$, Laplacian $\Delta'':= \overline{\partial} \overline{\delta} +\overline{\delta} \overline{\partial} $, harmonic forms, $L^2$-norms, Sobolev norms, etc.
\end{itemize}

For technical reason, 
we generalize the complex of (usual smooth) superforms $(\A_{\Trop}^{p,*}, $ $d'' \colon \A_{\Trop}^{p,*} \to \A_{\Trop}^{p,* + 1})$ 
to that of \emph{quasi-smooth superforms} 
$(\A_{\Trop,\qs}^{p,*},\overline{\partial} \colon \A_{\Trop,\qs}^{p,*} \to \A_{\Trop,\qs}^{p,* + 1})$.  
We also consider 
\emph{piece-wise quasi-smooth superforms} $\A_{\Trop,\pqs  }^{p,q-1} ( \Lambda )$, 
and those $\A_{\Trop,\pqs , \epsilon_1, (\epsilon_2) }^{p,q-1} ( \Lambda )$ 
    satisfying boundary condition(s) $\epsilon_1$ (and $\epsilon_2$).

The following is the main results of this paper.

\begin{thm}[the Hodge-Morrey-Friedrichs decompostion (\cref{the Hodge-Morrey-Friedrichs decomposition})]\label{Intro the Hodge-Morrey-Friedrichs decompostion}
  Let $\Lambda$ be a tropically compact projective tropical variety.
  Under \cref{assumption regularity of Derichlet potentials} (regularity of solutions of tropical Laplacian),
  we have an $L^2$-orthogonal decomposition 
   \begin{align*}
  \A_{\Trop,\pqs}^{p,q} ( \Lambda )
   & = \{ \overline{ \partial } \alpha \mid \alpha \in \A_{\Trop,\pqs , \t''_{\min} }^{p,q-1} ( \Lambda ) \}
     \oplus \{ \overline{ \delta } \beta \mid \beta \in \A_{\Trop,\pqs , \n''_{\min} }^{p,q+1} ( \Lambda ) \}  \\
    & \oplus \H_{D}^{p,q} ( \Lambda )  
 \oplus  \{ \overline{ \delta} \gamma \in \H^{p,q} ( \Lambda ) \cap \A_{\Trop,\pqs}^{p,q}(\Lambda) 
  \mid \gamma \in \A_{\Trop, \pqs, \n''_{\max} }^{p,q+1} ( \Lambda )  \}    .
  \end{align*}
where (under \cref{assumption regularity of Derichlet potentials}) we put 
\begin{align*}
  \H_{D}^{p,q} ( \Lambda ) 
 & := \{ \omega \in \A_{\Trop, \pqs, \t''_{\min}, \n''_{\max} }^{p,q} ( \Lambda ) \mid \overline{\partial} \omega = 0, \ \overline{\delta} \omega = 0  \}  , \\
  \H^{p,q} ( \Lambda ) \cap \A_{\Trop,\pqs}^{p,q}(\Lambda) 
 & :=
    \{ \omega \in \A_{\Trop, \pqs, \t''_{\max}, \n''_{\max} }^{p,q} ( \Lambda ) \mid \overline{\partial} \omega = 0, \ \overline{\delta} \omega = 0  \}     .
\end{align*}
\end{thm}

\begin{cor}[the Hodge isomorphism (\cref{the Hodge isomorphism})]\label{Intro the Hodge isomorphism}
  We have 
   $$  H_{\Trop,\pqs,\Dol, \t''_{\min}}^{p,q} (\Lambda) := H^q( ( \A_{\Trop,\pqs,\t''_{\min} }^{p,*} ( \Lambda ), \overline{\partial} )) \cong \H_D^{p,q} ( \Lambda ) .$$
\end{cor}

\begin{thm}[the K\"{a}hler package (\cref{hard Lefschetz} and \cref{Hodge-Riemann bilinear relation})]\label{Intro the Kahler package}
  Under \cref{assumption regularity of Derichlet potentials},
  when $\Lambda$ is  $\Q$-smooth in codimension $1$, 
  the K\"{a}hler package holds for  $\H_D^{p,q} ( \Lambda )$. 
\end{thm}

Loosely speaking, this paper can be devided into 
\begin{itemize}
  \item introducing new notion in this paper (Section 3, 4),
  \item Hodge theory similar to that of compact Riemann manifolds with smooth boundaries (\cite[Chapter 2]{Sch95}) (Section 5,6,7),
  \item the K\"{a}hler package, proved in a parallel way to that of compact K\"{a}hler manifolds (Section 8), and
  \item Density of piece-wise quasi-smooth superforms in Sobolev spaces with boundary conditions (Section 9) (technical part).
\end{itemize}

More precisely, the content of this paper is as follows.

In Section 3, 
we shall introduce new differential forms, (piece-wise) quasi-smooth superforms.
We also define $\R$-K\"{a}hler forms, $L^2$-norms, Sobolev $s$-norms, $\overline{*}$-operator, etc.
In Section 4, 
we shall define normal frames and 4 boundary conditions.

In Section 5, 
we shall prove an analog of Gaffney-G{\aa}rding's inequality (\cref{Gaffney-Garding's inequality}), 
the key proposition in this paper, 
which is an estimate of Sobolev $1$-norm.
It guarantees the existence and uniqueness of weak solutions of Laplacian in Section 6.
In Section 7, 
we shall prove the Hodge-Morrey-Friedrichs decomposition (\cref{Intro the Hodge-Morrey-Friedrichs decompostion}) and  the Hodge ismorphism (\cref{Intro the Hodge isomorphism})
under an assumption of regularity of weak solutions of Laplacian.
Section 6 and 7 are completely parallel to Hodge theory of compact Riemann manifolds with smooth boundaries.

In Section 8, 
we shall prove the K\"{a}hler package (\cref{Intro the Kahler package}) in the same way as compact K\"{a}hler manifolds.

In Section 9, 
we shall prove density (\cref{density H = W with boundary conditions2}) of piece-wise quasi-smooth superforms with fixed boundary condition(s) in the Sobolev $1$-space with the same condition(s).
This is used to prove Gaffney-G{\aa}rding's inequality (\cref{Gaffney-Garding's inequality}).
The reader can skip this technical part.

\medbreak
\noindent{\bfseries Acknowledgment.}\quad
The author is deeply grateful to Tetsushi Ito and Yuan-Pin Lee 
for their encouragement.
Part of this paper was studied during author's stay at National Cheng Kung University. He is grateful to Jen-Chieh Hsiao and National Cheng Kung University for their hospitality.

\section{Notation}\label{sec;not;term}

\begin{itemize}
  \item For a $\Z$-module $G$, we put $G_\R := G \otimes_\Z \R$.
  \item For simplicity, we put e.g.,
          $$ f(x):= f(x_1,\dots,x_n)$$
          $$ f(- \log \lvert z \rvert ):= f(- \log \lvert z_1\rvert,\dots, - \log \lvert z_n \rvert ).$$
  \item For a differential $(1,0)$-form $\varphi'$, we put $\varphi'':= \overline{\varphi'}$ the complex conjugate.   Similarly, for a vector $(1,0)$-field $E'$, we put $E'':= \overline{E'}$.
  \item For a finite set $I=(i_1,\dots,i_r) \in \Z^r $,
  we use multi-index to simplify notations.
  For example,
  \begin{align*}
   d (-\frac{i}{2  } \log z_I):=& 
   d (-\frac{i}{2 } \log z_{i_1})\wedge \dots
   \wedge 
   d (-\frac{i}{2} \log z_{i_r}).
  \end{align*}
  \item The notation $\sum'$ means a restricted sum.
  For example, $\sum'_{I \in \{1,\dots,n\}^p , J \in \{1,\dots,n\}^q}$ means 
    the sum for all $I =( i_1,\dots, i_p)  \in \{1,\dots,n\}^p$
      and $J =( j_1,\dots, j_q)  \in \{1,\dots,n\}^q$,
    with $i_1 < \dots < i_p$ 
      and $j_1 < \dots < j_q$.
  When there is no confusion, 
  we denote it simply by $ \sum_{I,J}'$.
  \item For a differential $p$-form $\omega$ and vector fields $v_1,\dots,v_s$ ($s \leq p$), 
    we put $\omega(v_1,\dots,v_s)$ the contraction, 
    i.e., the $(p-s)$-form given by 
     $$(w_1,\dots,w_{v-s}) \mapsto \omega (v_1,\dots,v_s,w_1,\dots,w_{v-s}) $$
  (see e.g., \cite{Sch95}).
\end{itemize}

Throughout this paper,
we fix 
  a free $\Z$-module $M $ of finite rank $n$. We put $N:=\Hom(M,\Z)$. We often identify $\N_\R \cong \R^n$.

We also fix
\begin{itemize}
  \item a smooth projective toric variety $T_{\Sigma}$ over $\C$ (see \cite{CLS11} for toric geometry)  and 
  \item a unimodular fan $\Sigma$ in $N_\R$ corresponding to it.  
\end{itemize}

In Section 4-9, 
we fix 
\begin{itemize}
  \item a $d$-dimensional tropically compact (projective) tropical variety $(\Lambda, (m_P)_{P \in \Lambda_{\max}}) $ 
  in the projective tropical toric variety $\Trop (T_{\Sigma})$ and
  \item  an $\R$-K\"{a}hler superform $\omega_{Kah}$ on $\Lambda$ (see \cref{definition of Kahler forms}).
\end{itemize}

\section{Differential forms}

In this section, we shall remind definition of tropical varieties,
introduce (piece-wise) quasi-smooth superforms,
and introduce $\R$-K\"{a}hler superforms,
which enable us to define 
$L^2$-norms, Sobolev $s$-norms, and some operators
on tropical varieties 
in the same way as compact K\"{a}hler manifolds.

Remind that there is a natural bijection between the cones $\sigma \in \Sigma$ and the torus orbits $O(\sigma)$ in $T_{\Sigma}$.
The torus orbit $O(\sigma)$ is isomorphic to the torus $ \Spec \C[M  \cap \sigma^{\perp}].$
We put $N_{\sigma}:= \Hom (M  \cap \sigma^{\perp} , \Z)$.
We put $T_{\sigma} := \bigcup_{\substack{ \tau \in \Sigma \\ \tau \subset \sigma}} O(\tau)$ the affine toric variety corresponding to a cone $\sigma$.
See \cite{CLS11} for toric geometry.

\subsection{Tropical varieties}\label{subsec;fan}
In this subsection,
we recall the compactification $\Trop(T_{\Sigma})$  of $N_\R \cong \R^n$ and tropical varieties in it.

We  put 
$\Trop(T_{\Sigma}) := \bigsqcup_{\sigma \in \Sigma } N_{\sigma ,\R}$ as a set.
We define a topology on $\Trop(T_{\Sigma})$  as follows.
We extend the canonical topology on $\R$  to that on $\R \cup \{ \infty \}$ so that  $(a, \infty]$ for $a \in \R$ are a basis of neighborhoods of $\infty$.
We also extend  the addition on $\R$  to that on $\R \cup \{ \infty \}$ by $a + \infty = \infty$ for  $a \in \R \cup \{\infty\}$.
 We consider the set of semigroup homomorphisms $\Hom (M \cap \sigma^{\vee}, \R \cup \{ \infty \}) $  as a topological subspace of $(\R \cup \{ \infty \})^{M \cap \sigma^{\vee}} $.
We define a topology on 
  $\Trop (T_\sigma ) := \bigsqcup_{ \substack{\tau \in \Sigma \\ \tau \subset \sigma  } } N_{\tau,\R}$ 
    by the canonical bijection
$$\Hom (M \cap \sigma^{\vee}, \R \cup \{ \infty \}) \cong \bigsqcup_{ \substack{\tau \in \Sigma \\ \tau \subset \sigma  } } N_{\tau,\R}.$$
 Then we define a topology on  $\Trop(T_{\Sigma}) = \bigsqcup_{\sigma \in \Sigma } N_{\sigma,\R}$ by glueing the topological spaces $ \Trop ( T_\sigma ) $ together.

A toric morphism $\varphi \colon T_{\Sigma'} \to T_{\Sigma}$
induces a natural map
$$ \Trop(T_{\Sigma'}) \to \Trop(T_{\Sigma}),$$
which is also denoted by $\varphi$.

We shall define polyhedral complexes in $\Trop(T_{\Sigma})$.
\begin{dfn}
A subset of $\R^r$ is called a \emph{rational polyhedron}
if it is the intersection of sets of the form
$$\{x \in \R^r \mid \langle x , a \rangle \leq b \} \ (a \in \Z^r, b \in  \R),$$
here $\langle x,a\rangle$ is the usual inner product of $\R^r$.
\end{dfn}
In this paper, we call rational polyhedron simply polyhedron.
\begin{dfn}
For a cone $\sigma\in \Sigma $ and a polyhedron  $C\subset N_{\sigma,\R}$,
we call its closure $P:=\overline{C}$ in  $ \Trop(T_{\Sigma})$ a  \emph{polyhedron}  in $\Trop(T_{\Sigma})$.
We put $\dim(P):=\dim(C)$.
\end{dfn}

Let $P \subset \Trop(T_{\Sigma})$ be a polyhedron.
We put $\sigma_P \in \Sigma $ the unique cone such that $P \cap N_{\sigma_P,\R} \subset P$ is dense.
A subset $Q$ of  $P$ in $\Trop(T_{\Sigma})$ is called a \emph{face} of $P$ if it is
the closure of the intersection $P^a \cap  N_{\tau,\R}$
in $ \Trop(T_{\Sigma})$
for some $ a \in \sigma_P\cap M$ and some cone $\tau \in \Sigma $,
where $P^a$ is the closure of
$$\{x \in P\cap  N_{\sigma_P,\R} \mid x(a) \leq  y(a) \text{ for any } y \in P \cap  N_{\sigma_P,\R} \}  $$
in $ \Trop(T_{\Sigma})$.
A finite collection $\Lambda$ of   polyhedra  in $ \Trop(T_{\Sigma})$ is called  a   \emph{polyhedral complex} if it satisfies the following two conditions.
\begin{itemize}
\item For all $P \in \Lambda$, each face of $P$ is also in $\Lambda$.
\item For all $P,Q \in \Lambda$, the intersection $P \cap Q$ is a face of $P$ and $Q$.
\end {itemize}

Let $\Lambda$ be a polyhedral complex of pure dimension $d$.
We put $\Lambda_{\max} \subset \Lambda$ be the subset of maximal dimensional (i.e., $d$-dimensional) polyhedra.
We put $\Lambda_{\max - 1} \subset \Lambda$ be the subset of $(d-1)$-dimensional polyhedra.

\begin{dfn}\label{definition of tropical variety}
  A (projective) \emph{tropcial variety} $(\Lambda,(m_P)_{P \in \Lambda_{\max}})$ of dimension $d$ in $\Trop(T_{\Sigma})$
  consists of 
  a polyhedral complex $\Lambda$
  of pure dimension $d$
  and 
  integers $m_P \in \Z$
  such that 
  for any cone $Q \in \Lambda_{\max - 1 }$,
  we have 
  $$\sum_{ \substack{P \in \Lambda_{\max} \\ Q \subset P , \sigma_Q = \sigma_{P}} } m_{P} n_{Q,P} =0 ,$$
  where $n_{Q,P} \in N_{\sigma_Q, \R}/ \Tan(Q \cap N_{\sigma_Q,\R})$
  is the primitive vector spanning the image of $P \cap N_{\sigma_Q,\R}$.
  (Here $\Tan ( Q \cap N_{\sigma_Q,\R})$ is the tangent space considered as a subspace  of $N_{\sigma_Q, \R }$ in the natural way.)
\end{dfn}

By abuse of notation, 
we simply say a tropical variety $\Lambda$ 
        or a tropical variety $X := \bigcup_{P \in \Lambda_{\max}} P \subset \Trop (T_{\Sigma})$.

\begin{dfn}
  In this paper, 
we say that a polyhedron $P\subset \Trop(T_{\Sigma})$
is  \emph{tropically compact}
if
 \begin{itemize}
  \item ($P$ is compact, which is always the case since $T_{\Sigma}$ is proper)
\item $ P \cap N_\R \neq \emptyset$,
  \item  $P \subset \Trop(T_{\sigma})$ for some cone $\sigma \in \Sigma$,
  and 
  \item  for any $\tau \in \Sigma$  
    and $x \in P \cap N_{\tau,\R}$,
  there exists an open neighborhood $U_x \subset \Trop(T_{\tau})$ of $x$
  such that 
  $$P \cap N_\R \cap U_x =((P \cap N_{\tau,\R}) \times \Tan \tau) \cap U_x,$$
  where we fix an identification $N_\R  \cong N_{\tau,\R} \times \Tan \tau$.
 \end{itemize}

  We say that a tropical variety $\Lambda$ is \emph{tropically compact}
  if every maximal dimensional polyhedron $P \in \Lambda_{\max}$
  is tropically compact.
\end{dfn}
In this case, 
we put $\Lambda^{\ess} \subset \Lambda$ be the subset of  polyhedra intersecting with $N_\R \cong \R^n$ 
and $\Lambda_{\max -1 }^{\ess} := \Lambda_{\max - 1} \cap \Lambda^{\ess}$.

\begin{dfn}\label{definition of locally irreducible tropical varieties}
  We say that a tropically compact tropical variety $(\Lambda,(m_P)_{P \in \Lambda_{\max}})$ is \emph{$\Q$-smooth in codimension $1$}
  if 
  for a polyhedron $Q \in \Lambda_{\max - 1 }^{\ess}$
  and integers $m'_{Q,P} \in \Z $ ($ P \in \Lambda_{\max} $ containing $Q$)
  satisfying
  $$\sum_{ \substack{P \in \Lambda_{\max} \\  Q \subset P }} m'_P n_{Q,P} =0 ,$$
  there exists 
  a rational number $a \in \Q$ 
  such that 
  $m'_{Q,P} = a m_P$ 
  for all $P \in \Lambda_{\max} $ containing  $ Q$. 
\end{dfn}

A smooth (i.e., locally matroidal) tropical variety is $\Q$-smooth in codimension $1$.

\subsection{Quasi-smooth superforms}\label{subsection superforms and tropical Dolbeault cohomology}
In this subsection, we shall introduce (piece-wise) quasi-smooth superforms.

First, we shall define quasi-smooth superforms on tropical toric varieties.
We put 
$$ - \log \lvert \cdot \rvert \colon 
T_{\Sigma}(\C) \to \Trop(T_{\Sigma})$$
the coordinate-wise map,
i.e., for any cone $\sigma \in \Sigma$,
$$ - \log \lvert \cdot \rvert|_{O(\sigma)(\C)} \colon 
O(\sigma)(\C) = \Hom (M \cap \sigma^{\perp},\C^{\times}) \ni z \mapsto (- \log \lvert \cdot \rvert) \circ z \in N_{\sigma,\R} = \Hom (M \cap \sigma^{\perp},\R).$$

\begin{dfn}\label{definition quasi-smooth superforms}
Let $U \subset \Trop(T_{\Sigma})$ be an open subset.
A \emph{quasi-smooth $(p,q)$-superform} on $U$ or on $(- \log \lvert \cdot \rvert)^{-1}(U)$ 
is 
  a smooth differential $(p,q)$-form $\omega$ on 
  a complex manifold $(- \log \lvert \cdot \rvert)^{-1}(U)$ of the form
  \begin{align*}
    \omega
    = \sum'_{\substack{I = (i_1,\dots,i_p) \\ J=(j_1,\dots,j_q)} }
	   \alpha_{I,J} (-  \log \lvert z \rvert )
     d ( - \frac{i}{2 } \log  z_I )\wedge  d ( - \frac{1}{2} \log  \overline{z_J}  ),
  \end{align*}
  for some $\R$-valued functions 
    $$\alpha_{I,J} (-  \log \lvert z \rvert ) := \alpha_{I,J} (-  \log \lvert z_1 \rvert, \dots, -  \log \lvert z_n \rvert ) ,$$
  where $z_1, \dots,z_n \in M$ is a basis considered as a holomorphic coordinate of $(\C^{\times})^{n}$,
  and 
    $$d ( - \frac{i}{2 } \log  z_I ):= 
       d ( - \frac{i}{2 } \log  z_{i_1} ) \wedge \dots \wedge 
       d ( - \frac{i}{2 } \log  z_{i_p} )$$
  (similarly for $d ( - \frac{1}{2 } \log  \overline{z_J} )$).

 We put 
 $ \A_{\Trop, \qs}^{p,q} ( ( - \log \lvert \cdot \rvert )^{-1} ( U) ) $  
 the set of quasi-smooth superforms on $U$. 
 We also denote it by $\A_{\Trop, \qs}^{p,q} (   U ) $.
\end{dfn}

Next, we shall define quasi-smooth superforms on 
 a $d$-dimensional tropically compact polyhedron $P \subset \Trop (T_{\Sigma})$.
There is 
  a basis $x_1,\dots,x_n$ of $M$ (which gives a coordinate of $N_\R \cong \R^n$)
  such that the affine span of $P \cap \R^n$ is 
    $$ \{ (x_1,\dots,x_n) \in \R^n \mid x_i = a_i \ ( d +1 \leq i \leq n ) \}$$
  for some $a_i \in \R \ ( d +1 \leq i \leq n)$.
We put 
  $z_1,\dots,z_n$ the basis $x_1,\dots,x_n$ as a holomorphic coordinate of $(\C^{\times})^n$.
We put $\exp ( - P)_\C \subset T_{\Sigma}(\C)$ the closure of 
  $$ \{z \in  ( \C^{\times} )^n \mid  z_i = e^{-a_{i} } \ ( d+1 \leq i \leq n ) \} 
     \cap  (- \log \lvert \cdot \rvert )^{-1} (P)  .$$
  It is also the closure of a $d$-dimensional complex manifold $\relint ( \exp (-P)_\C)$.

\begin{dfn}
For an open subset $V$ of $P $, 
a \emph{quasi-smooth $(p,q)$-superform} on $V$ or on $\exp ( - P)_\C  \cap  (- \log \lvert \cdot \rvert )^{-1} ( V ) $ is a smooth differential $(p,q)$-forms on $\exp ( -P)_\C \cap  (- \log \lvert \cdot \rvert )^{-1} ( V ) $ (in a natural sense, see \cref{some objects on exp - P C })
which is 
  the restriction of a quasi-smooth $(p,q)$-superform on an open neighborhood of  $\exp ( - P)_\C  \cap  (- \log \lvert \cdot \rvert )^{-1} ( V ) $ in $T_{\Sigma}(\C)$.
  
We put
  $$ \mathscr{A}_{\Trop, \qs }^{p,q}( \exp ( - P)_\C  \cap  (- \log \lvert \cdot \rvert )^{-1} ( V )  )   $$ 
  the subset of quasi-smooth $(p,q)$-superform.
  We also denote it by  $\mathscr{A}_{\Trop, \qs }^{p,q}( V)$.
\end{dfn}

Finally, we shall define quasi-smooth superforms on a tropically compact tropical variety $(\Lambda, (m_P)_{P \in \Lambda_{\max}})$ of dimension $d$.
For an open subset $\Omega \subset X := \bigcup_{P \in \Lambda_{\max} } P$,
we put 
$$\A_{\Trop,\pqs}^{p,q}(\Omega) :=
\bigoplus_{P \in \Lambda_{\max}} \A_{\Trop,\qs}^{p,q}( \exp ( - P)_\C  \cap ( - \log \lvert \cdot \rvert )^{-1} ( \Omega ) ) ,$$
the set of \emph{piece-wise quasi-smooth superforms}.
We put $\A_{\Trop,\pqs}^{p,q}(\Lambda) := \A_{\Trop,\pqs}^{p,q}(X)$.

\begin{dfn}
A \emph{quasi-smooth $(p,q)$-superform} on $\Omega$  is a piece-wise quasi-smooth $(p,q)$-forms on $\Omega$ 
which is the restrictions of a quasi-smooth $(p,q)$-superforms on an open neighborhood of  
$$ \bigcup_{P \in \Lambda_{\max}} \exp (-P)_\C \cap  (- \log \lvert \cdot \rvert )^{-1} (\Omega)$$ in $T_{\Sigma}(\C)$.

We put 
  $\mathscr{A}_{\Trop,\qs}^{p,q}(\Omega)$
  the subset of quasi-smooth $(p,q)$-superforms.
We put $C_{\Trop,\qs}^{\infty}:= \A_{\Trop,\qs}^{0,0}$.
\end{dfn}
  
 We put 
 $$\overline{\partial} \colon \A_{\Trop, \pqs}^{p,q} (   \Omega )  
     \to \A_{\Trop, \pqs}^{p,q + 1} (   \Omega ) $$
  the formal sum of  
 the restrictions of the anti-holomorphic differential 
 $\overline{\partial} $
  of the (usual) smooth differential $(p,q)$-forms on open neighborhoods of  $\exp (-P)_\C \cap  (- \log \lvert \cdot \rvert )^{-1} (\Omega)$ ($ P \in \Lambda_{\max}$) in $T_{\Sigma}( \C)$.
  It gives a complex of sheaves 
  $(\A_{\Trop,\qs}^{p,*}, \overline{\partial}) $ of quasi-smooth superforms on $X$.

For a toric morphism $\varphi \colon T_{\Sigma'} \to T_{\Sigma}$
        between smooth toric varieties
    and tropical varieties $X \subset \Trop(T_{\Sigma})$ 
        and 
        $X' \subset \Trop(T_{\Sigma'})$
such that 
  $\varphi(X') \subset X$,
there is 
a natural pull-back map
  $$\varphi^* \colon \A_{\Trop,\qs,X }^{p,q} \to \A_{\Trop,\qs,X' }^{p,q}$$
of sheaves of quasi-smooth superforms.

\begin{rem}\label{comparison of smooth superforms and quasi-smooth superforms}
  There is a natural morphism 
    from the complex of sheaves of superforms $(\A_{\Trop}^{p,*}, d'')$ (see \cite[Definition 2.10]{JSS19} (where notation is different) for definition)
    to the complex of sheaves of quasi-smooth superforms $(\A_{\Trop,\qs}^{p,*}, \overline{\partial})$ 
    on $X$
    given by
    \begin{align*}
       \sum'_{I,J} f_{I,J}(x) d'x_I \wedge d'' x_J 
       & \mapsto
      i^p \sum'_{I,J} f_{I,J}( - \log \lvert z \rvert ) 
       \partial ( - \log \lvert z_I \rvert ) \wedge 
       \overline{\partial} ( - \log \lvert z_J \rvert ) \\
       & =
      \sum'_{I,J} f_{I,J}( - \log \lvert z \rvert ) 
       d ( - \frac{i}{2} \log  z_I ) \wedge 
       d ( - \frac{1}{2} \log \overline{z_J} ). 
    \end{align*}
  The restriction of this morphism to $X \cap \R^n$ is an isomophism of complexes of sheaves.
\end{rem}

\begin{rem}\label{some objects on exp - P C }
We shall define vector fields and smooth forms on  an open subset $ U \subset \exp ( -P_\C )$ in an obvious way.
\begin{itemize}
  \item
For $z \in \exp ( -  P_\C  )$,
we put 
  $$T_z \exp ( -  P_\C  ) := T_z V_z   $$
  the tangent space at $z$ of 
a $d$-dimensional complex manifold $V_z$
which is the intersection of an open neighborhood of $z \in T_{\Sigma} ( \C )$
and
$$ \overline{  \{z \in  ( \C^{\times} )^n \mid  z_i = e^{-a_{i} } \ ( d+1 \leq i \leq n ) \} } $$
(here we use the notation  ($z_i, a_i$) used in definition of $\exp ( - P)_\C $).
We also put 
  $$T U:= \bigsqcup_{ z \in U } T_z \exp ( - P)_\C   .$$
A \emph{vector field} on $U$ 
  is a map $U \to  T U $
     which is locally the restriction of a vector field 
        of the manifold $V_z$.
We put $\Gamma (  T U)$ the set of vector fields.

  \item
A \emph{smooth function} on $U$ is the restriction of a smooth function defiend on a neighborhood.
We put $C^{\infty} (U)$ (resp. $C_0^{\infty} (U)$) the set of smooth functions (resp. the set of compactly supported smooth functions).
When we  explicitly indicate the values of smooth functions,
we denote the set of them by $C_\R^{\infty} (U)$ or $C_\C^{\infty} (U)$.
  \item
A \emph{smooth differential $p$-form} on $U$
is a map 
$$ \bigsqcup_{z \in U } \wedge^p T_z \exp ( - P)_\C  \to \C$$
which is a restriction of a smooth differential $p$-form defined on a neighborhood.
We put $\A^p ( U)$ the set of smooth differential $p$-forms.
  \item
We also define 
  the set of \emph{smooth differential $(p,q)$-forms} $\A^{p,q} ( U)$ 
 and  
  the set of \emph{holomorphic} and \emph{anti-holomorphic vector fields} 
    $\Gamma ( T^{1,0} U ) $ 
   and 
    $\Gamma ( T^{0,1} U ) $
by using the complex structure of $V_z$.
\end{itemize}
\end{rem}

\subsection{Cohomology of quasi-smooth superforms}
A tropical analog of the Poincar\'{e} lemma was first proved by Jell  \cite[Theorem 2.16]{Jel16A} for superforms on tropical varieties in $\R^n$.
It was generalized by Jell-Shaw-Smacka \cite[Theorem 3.16]{JSS19} to superforms on tropical varieties in smooth tropical toric varieties.
In this subsection, we prove it for quasi-smooth superforms on 
 a tropically compact tropical variety $\Lambda$.
 We put
 $X:= \bigcup_{P \in \Lambda_{\max}}P  \subset \Trop(T_{\Sigma})$.

We put
$$  H_{\Trop,\Dol, \qs}^{p,q}(X)
  := H^q( (\A_{\Trop,\qs}^{p,*} (X), \overline{\partial})).$$

\begin{prp}\label{exactness of complex of quasi-smooth forms}
  For any $p$,
  the complex $(\A_{\Trop, \qs}^{p,*}, \overline{\partial})$ is a fine resolution of the sheaf
  $\Ker (\overline{\partial} \colon \A_{\Trop, \qs}^{p,0} \to \A_{\Trop, \qs}^{p,1})$ on $X$.
\end{prp}

\begin{rem}
  Let 
  $$H_{\Trop,\Dol}^{p,q}(X)
  := H^q( (\A_{\Trop}^{p,*} (X), d'')).$$
  be the usual tropical Dolbeault cohomology.
  It is easy to see that the natural inclusion of sheaves 
  $\A_{\Trop}^{p,*} \to \A_{\Trop, \qs}^{p,*}$
  (\cref{comparison of smooth superforms and quasi-smooth superforms})
  induces  an isomorphism
$$\Ker (d''\colon \A_{\Trop}^{p,0} \to \A_{\Trop}^{p,1}) \cong
\Ker ( \overline{\partial} \colon \A_{\Trop, \qs}^{p,0} \to \A_{\Trop, \qs}^{p,1}).$$
Hence by the Poincar\'{e} lemmas (\cite[Theorem 3.16]{JSS19} and \cref{exactness of complex of quasi-smooth forms}),
we have
    $$ H_{\Trop,\Dol}^{p,q}(X) \cong H_{\Trop,\Dol,\qs}^{p,q}(X)$$
\end{rem}

\begin{proof}
There is a partition of unity by $\A_{\Trop}^{0,0}$ on $X$ (\cite[Lemma 2.7]{JSS19}).
In paricular, the sheaf $\A_{\Trop, \qs}^{p,q}$ is fine.

Let $\sigma \in \Sigma$ be a cone.
Then for a $\overline{\partial}$-closed quasi-smooth superform $\omega$ on $X \cap \Trop (T_\sigma )$
and $x \in X \cap N_{\sigma,\R}$, 
in the same way as \cite[Section 0.2]{GH78},
there exists a quasi-smooth superform $\eta_1$
such that 
$$\omega - \overline{\partial} \eta_1 = \sum'_{\substack{I  \in \{1,\dots, n \}^p   \\ J \in \{1,\dots, n-\dim \sigma\}^q }}
           f_{I,J} ( - \log \lvert z \rvert ) d( - \frac{i}{2} \log z_I ) \wedge d ( -\frac{1}{2} \log \overline{z_J})$$
in a neighborhood of $x \in X \cap \Trop (T_\sigma )$
for some functions $f_{I,J}$,
where 
  $z_i$ is the coordinate of $T_{\sigma} \cong (\C^{\times})^{n-\dim \sigma} \times \C^{\dim \sigma}$.
Then 
we can find a quasi-smooth superform $\eta$ such that $\omega = \overline{\partial} \eta$
in the same way as \cite[Subsection 2.2]{Jel16A}
(but using a map onto $\{x\} \times (\R \cup \{\infty\})^{\dim \sigma}$ instead of a constant map to a point, here we identify 
$$\Trop(T_{\sigma} ) \cong N_{\sigma,\R} \times (\R \cup \{\infty\})^{\dim \sigma} ).$$
\end{proof}

\subsection{$\R$-K\"{a}hler forms, $L^2$-norms, and Sobolev norms}
Lagerberg (\cite{Lag11}) introduced an analog of K\"{a}hler forms,
 \emph{$\R$-K\"{a}hler superforms} on $\R^n$.
In this subsection, we shall introduce their generalizations to tropical varieties.
We also define $L^2$-norms, Sobolev $s$-norms, $\overline{*}$-operator, codifferential $\overline{\delta} := - \overline{*} \overline{\partial}\overline{*}$, and Laplacian $\Delta'':=  \overline{\delta} \overline{\partial}+ \overline{\partial}\overline{\delta} $ in the same way as K\"{a}hler geometry.
We will use basic facts of K\"{a}hler geometry freely.
See \cite{GH78} and \cite{Wei58} for K\"{a}hler geometry.

Let $\Lambda$ be a tropically compact tropical variety of dimension $d$.

\begin{dfn}\label{definition of Kahler forms}
  A quasi-smooth superform on $ \Lambda $
  is called an \emph{$\R$-K\"{a}hler superform} 
  if it is the restriction of a  K\"{a}hler form on $T_{\Sigma} (\C)$.
\end{dfn}

\begin{exam}
  The K\"{a}hler form $\omega$ of the Fubini-Study metric on $\P^r(\C)$ is an $\R$-K\"{a}hler superform.
  Its pull-back under a toric closed immersion $T_{\Sigma} \to \P^r ( \C)$ is also an $\R$-K\"{a}hler superform.
\end{exam}

Let 
$$\omega_{Kah}= i \sum_{i,j} \omega_{Kah, i,j}( - \log \lvert z \rvert ) d ( - \frac{1}{2} \log z_i ) \wedge d ( - \frac{1}{2} \log \overline{z_j} )   \in \A_{\Trop,\qs}^{1,1}(T_{\Sigma}(\C)) $$ 
be an $\R$-K\"{a}hler superform,
and $P \in \Lambda_{\max}$ be a polyhedron.
We put $g:= g_{\omega_{Kah}}$ the induced Riemann metric on $T_{\Sigma}(\C)$.
It induces a map
$$ g 
   \colon \bigsqcup_{a \in \exp ( - P)_\C  } T_{\R,a} ( \exp ( - P)_\C  )
          \times 
          \bigsqcup_{a \in \exp ( - P)_\C  } T_{\R,a} ( \exp ( - P)_\C  )
   \to \R.$$
By the Gram-Schmidt process, 
we have 
  $$\omega_{Kah}|_{\exp ( - P)_\C  }
    = i \sum_{i=1}^d \varphi_i' \wedge \varphi_i'' \in \A_{\Trop,\qs}^{1,1}(\exp ( - P)_\C  )$$ 
  for some $  \varphi_i' \in \A^{1,0}(\exp ( - P)_\C  ) $, 
  where 
    we put 
    $\varphi_i'' := \overline{\varphi_i'} \in \A^{0,1}(\exp ( - P)_\C  ) $
    the complex conjugate of $\varphi_i'$.

\begin{rem}\label{quasi-smooth superform orthonormal basis}
  In the above discussion,
  because of singularlity of $- \log \lvert z_i \rvert $ at (subsets of) the toric boundary of $T_{\Sigma} ( \C )$,
  it is NOT ALWAYS possible 
    to take $\varphi_1', \dots, \varphi_d'$ 
    so that 
      $ i \varphi_1', \dots, i \varphi_d' $ are \emph{quasi-smooth superforms}, 
      and
  it is ALWAYS possible 
  to take $  \phi_i' \in \A^{1,0}(\exp ( - P)_\C   \cap (\C^{\times})^n ) $ 
  such that 
    $$\omega_{Kah}|_{\exp ( - P)_\C  \cap (\C^{\times})^n }
      = i \sum_{i=1}^d \phi_i' \wedge \phi_i'' \in \A_{\Trop,\qs}^{1,1}(\exp ( - P)_\C  \cap (\C^{\times})^n )$$ 
  and 
    $i \phi_1', \dots, i \phi_d'$ are quasi-smooth superforms. 
\end{rem}

For 
\begin{align*}
  \alpha & = \sum'_{I,J} \alpha_{I,J} \varphi_I' \wedge \varphi_J'' \in \A^{p,q}(\exp ( - P)_\C  ), \\
  \beta & = \sum'_{I,J} \beta_{I,J} \varphi_I' \wedge \varphi_J'' \in \A^{p,q}(\exp ( - P)_\C  )
\end{align*}
  ($ \alpha_{I,J}, \beta_{I,J} \in C^{\infty} ( \exp ( - P)_\C  ) $, 
  $\varphi_I':= \varphi_{i_1}' \wedge \dots \wedge \varphi_{i_p}'$,
  $\varphi_J'':= \varphi_{j_1}'' \wedge \dots \wedge \varphi_{j_q}''$), 
we put 
  $$(\alpha,\beta) := \sum'_{I,J} \alpha_{I,J} \overline{\beta_{I,J}} \in C^{\infty} ( \exp ( - P)_\C  ) .  $$

\begin{lem}\label{ alpha beta also quasi-smooth}
  When $\alpha$ and $\beta$ are quasi-smooth superforms,
  the function $ ( \alpha, \beta )$ is in $ C_{\Trop,\qs}^{\infty}(\exp ( -P)_\C ) $.
\end{lem}
\begin{proof}
  This follows from 
    \cref{quasi-smooth superform orthonormal basis} and a basic fact on K\"{a}hler geometry 
    that
    we can use $\phi_i'$ as in \cref{quasi-smooth superform orthonormal basis} instead of $\varphi_i'$ to compute $( \alpha, \beta )$.
\end{proof}

We put 
  $$\langle \alpha , \beta \rangle_{L^2 (P) }
  := \langle \alpha , \beta \rangle_{L^2 ( \relint ( \exp ( - P)_\C  ) ) }
  :=  \int_{\relint ( \exp ( - P)_\C  ) } (\alpha,\beta) \omega_{Kah,d} ,$$
  where $\omega_{Kah,d} := \frac{ \wedge^d \omega_{Kah}}{d !}$ 
        is the volume form on 
              the K\"{a}hler manifold $(\relint ( \exp ( - P)_\C  ) , \omega_{Kah} )$.
We put 
$$L^{p,q,2}(\relint ( \exp ( - P)_\C  ) ) \quad
  ( \text{resp.} \ L_{\Trop}^{p,q,2}(\relint ( \exp ( - P)_\C  ) ) ) $$
the closure of $\A^{p,q}( \exp ( - P)_\C  )$
 (resp. $\A_{\Trop,\qs}^{p,q}( \exp ( - P)_\C  )$)
  with respect to the \emph{$L^2$-norm} 
  $$\alpha \mapsto \Arrowvert \alpha \Arrowvert_{L^2 (P) } 
  := \Arrowvert \alpha \Arrowvert_{L^2 (\relint ( \exp ( - P)_\C  ) ) } 
  :=\langle \alpha,\alpha\rangle_{L^2 ( \relint ( \exp ( - P)_\C  ) ) }^{\frac{1}{2}} . $$ 

  Let $(E_1, \dots, E_{2d}) $ be a $g$-orthonormal frame on $\exp ( - P)_\C $, i.e., a tuple  of smooth vector fields $ E_i \in \Gamma(T_\R \exp ( - P)_\C  )$ 
  such that 
  $$g(E_i,E_j)= \delta_{i,j} $$ 
  on $ \exp ( - P)_\C .$
  For $s \in \Z_{\geq 0}$, 
  we define  
    the Sobolev space $W^{p,q,s,2}(\relint ( \exp ( - P)_\C  ) )$ as the closure of $\A^{p,q}( \exp ( - P)_\C  )$ 
   with respect to the Sobolev $s$-norm $\Arrowvert \cdot  \Arrowvert_{W^{s,2}( P ) }$ given by 
  $$ \Arrowvert \omega \Arrowvert_{W^{s,2}( P ) }^2
   := \Arrowvert \omega \Arrowvert_{W^{s,2}( \relint ( \exp ( - P)_\C  ) ) }^2
    := \sum_{i = 0}^s \sum_{L=(l_1,\dots, l_i ) \in \{1, \dots, 2d \}^i } 
      \Arrowvert \nabla_{E_{l_1}} \dots \nabla_{E_{l_i}} \omega \Arrowvert_{L^2 ( P ) }^2  ,$$
   where $\nabla$ is the Levi-Civita connection on the Riemann manifold $(\relint (\exp ( - P)_\C  ) , g)$.
We put 
 $$W_{\Trop}^{p,q,s,2}(\relint ( \exp ( - P)_\C  ) ) 
  := W^{p,q,s,2}(\relint ( \exp ( - P)_\C  ) )
  \cap L_{\Trop}^{p,q,2}(\relint ( \exp ( - P)_\C  ) ) .$$
This is a Hilbert subspace of $ W^{p,q,s,2}(\relint ( \exp ( - P)_\C  ) ).$

\begin{rem}
  The Sobolev $s$-norm depends on the choice of a $g$-orthonormal frame $(E_1, \dots, E_{2d}) $ on $\exp ( - P)_\C $ for $s \in \Z_{\geq 2}$, but its equivalent class is independent of the choice,  see \cite[Section 1.3]{Sch95}.
\end{rem}

\begin{rem}\label{comaparison of Sobolev norm with Euclid metric}
  Since $P \in \Lambda_{\max}$ is tropically compact, 
  there is 
    a cone $\sigma \in \Sigma$ 
  such that 
    for a suitable choice of isomorphism 
      $T_{\sigma} ( \C ) \cong \C^{\dim \sigma} \times ( \C^{\times} )^{ n - \dim \sigma }$,
  we have 
    $$ \exp ( - P)_\C  \subset \C^{\dim \sigma} \times ( \C^{\times} )^{ d - \dim \sigma } \times \{(a_{ d+1 },\dots, a_n)\} $$
    ($a_i \in \C^{\times}$).
  Since 
    $\omega_{Kah}$ is defined on right hand side
    and
    $\exp ( - P)_\C$ is compact (in particular, bounded), 
  our Sobolev $s$-norm on $\exp ( - P)_\C $ 
  is equivalent to that given by the Euclid metric on 
    $ \C^{\dim \sigma} \times ( \C^{\times} )^{ d - \dim \sigma } \subset \C^d$.
\end{rem}

\begin{lem}\label{L2 superforms}
  \begin{align*}
    & L_{\Trop}^{p,q,2}(\relint ( \exp ( - P)_\C  ) )  \\
    = & \bigg\{  \sum'_{\substack{I \in \{1,\dots,d\}^p \\J  \in \{1,\dots,d\}^q }} \alpha_{I,J} ( - \log \lvert z \rvert )
           d ( - \frac{i}{2} \log z_I)   \wedge d ( - \frac{1}{2} \log \overline{z_J} ) \in L^{p,q,2}(\relint ( \exp ( - P)_\C  ) )  \\
     &     \bigg|  \alpha_{I,J} ( - \log \lvert z \rvert ) \text{ are } \R\text{-valued functions}  \bigg\} ,
  \end{align*}  
  where $z_1,\dots,z_d \in M$ are as in definition of $\exp (-P)_\C$.
\end{lem}
\begin{proof}
  It suffices to show that a differential form $\alpha$ contained in the right hand side is also contained in the left hand side.
  For simplicity, we assume $p=q=0$. The other case is similar.
  Then there exist $\alpha_i (z) \in C_\R^{\infty} ( \exp ( - P)_\C  )$ 
  converging to $ \alpha (- \log \lvert z \rvert )$
  in the $L^2$-topology
  such that 
    $$\supp \alpha_i \subset \exp ( - P)_\C  \cap (\C^{\times})^n.$$
  We put 
    $$\alpha_i' ( - \log \lvert z \rvert ) 
      := \frac{1}{( 2 \pi )^d} 
         \int_{\theta  = ( \theta_1, \dots, \theta_d ) \in [0,2 \pi]^d} 
          \alpha_i ( \lvert z_1 \rvert  e^{i \theta_1}, \dots, \lvert z_d \rvert  e^{i \theta_d}) 
           d \theta_1 \wedge \dots \wedge d \theta_d  
          .$$
  Then by Fubini's theorem and H\"{o}lder's inequality,
    functions $\alpha_i' ( - \log \lvert z \rvert ) \in C_{\Trop,\qs}^{\infty} ( \exp ( - P)_\C  )$ converge to $\alpha $ in the $L^2$-topology.
\end{proof}

As usual, we define $\overline{*}$-operator 
$$\overline{*} \colon L^{p,q,2}( \relint ( \exp ( - P)_\C  ) ) \to L^{d-p,d-q,2}( \relint ( \exp ( - P)_\C  ) )$$
by 
$$ (\alpha ,\beta) \omega_{Kah,d} = \alpha \wedge \overline{*}\beta.$$
where we extend $(\cdot,\cdot)$ to $L^{p,q,2}( \relint ( \exp ( - P)_\C  ) )$.
We have $\overline{*} \overline{*} \beta = (-1)^{p+q} \beta.$

\begin{lem}
 The $\overline{*}$-operator induces an isomorphism 
$$ \overline{*} \colon \A_{\Trop,qs}^{p,q}( \exp ( - P)_\C   ) \cong \A_{\Trop,qs}^{d-p,d-q}(  \exp ( - P)_\C   ).$$
\end{lem}
\begin{proof}
  Let $ \beta \in \A_{\Trop,qs}^{p,q}( \exp ( - P)_\C   )$.
  Then $\overline{*} \beta |_{\exp ( - P)_\C  \cap (\C^{\times})^n }$ is a quasi-smooth superform by \cref{ alpha beta also quasi-smooth}
  for quasi-smooth superforms $\alpha := i^p \phi_I' \wedge \phi_J''$ (with respect to all $I$ and $J$), 
  where $\phi_i'$ are as in \cref{quasi-smooth superform orthonormal basis}. Hence $\overline{*} \beta $ is a quasi-smooth superform.
\end{proof}

Since the $\overline{*}$-operator is $L^2$-isometric and $W^{s,2}$-isometric, we have the following.
\begin{cor}
  The $\overline{*}$-operator induces isomorphisms
    $$\overline{*} \colon L_{\Trop}^{p,q,2}( \relint ( \exp ( - P)_\C  ) ) 
    \cong L_{\Trop}^{d-p,d-q,2}( \relint ( \exp ( - P)_\C  ) )$$
   and 
    $$\overline{*} \colon W_{\Trop}^{p,q,s,2}( \relint ( \exp ( - P)_\C  ) ) 
    \cong W_{\Trop}^{d-p,d-q,s,2}( \relint ( \exp ( - P)_\C  ) ).  $$
\end{cor}

\begin{lem}
  Let 
    $s \in \Z_{ \geq 1}$.
  The continuous linear operator 
    $$\overline{\partial} \colon 
      W^{p,q,s,2}( \relint ( \exp ( - P)_\C  ) ) 
    \to W^{p,q+1,s-1,2}( \relint ( \exp ( - P)_\C  ) )  $$
    induces a continuous linear operator
    $$\overline{\partial} \colon 
      W_{\Trop}^{p,q,s,2}( \relint ( \exp ( - P)_\C  ) ) 
    \to W_{\Trop}^{p,q+1,s-1,2}( \relint ( \exp ( - P)_\C  ) ) . $$
\end{lem}
\begin{proof}
  Since $\overline{\partial}$ is  continuous and it induces 
    $$\overline{\partial} \colon 
      \A_{\Trop,\qs}^{p,q}(  \exp ( - P)_\C   ) 
    \to \A_{\Trop,\qs}^{p,q+1}( \exp ( - P)_\C   ) , $$
  it suffices to show that 
  for $\omega \in 
      W_{\Trop}^{p,q,s,2}( \relint ( \exp ( - P)_\C  ) ) $
      and $x \in P \cap \R^n$,
  there exist $f_x \in \A_{\Trop,\qs}^{p,q}(  \exp ( - P)_\C   ) $
    and an open neighborhood $U_x \subset P \cap \R^n$ of $x$
    such that 
    the Sobolev $s$-norm
    $$  \Arrowvert f_x - \omega \Arrowvert_{ 
      W^{s,2}( \relint ( \exp ( - P)_\C  )  \cap (- \log \lvert \cdot \rvert )^{-1} ( U_x ) ) } $$
      on $\relint ( \exp ( - P)_\C  )  \cap (- \log \lvert \cdot \rvert )^{-1} ( U_x )$
      (defined using a $g$-orthonormal frame on $\exp ( - P)_\C $)
    is small.
  This easily follows from the same method as \cref{L2 superforms}.
\end{proof}

We put 
$$\overline{\delta} := - \overline{*} \overline{\partial} \overline{*} \colon W_{\Trop}^{p,q,s,2}(\relint ( \exp ( - P)_\C  )  ) \to W_{\Trop}^{p,q-1,s-1,2}(\relint ( \exp ( - P)_\C  )  )$$
and
$\Delta'' := \overline{\delta} \overline{\partial} + \overline{\partial} \overline{\delta}$.

We define similar norms, spaces, and operators on $\Lambda$ as follows.
Let $\alpha = (\alpha_P)_{P \in \Lambda_{\max}}, $ $\beta = (\beta_P)_{P \in \Lambda_{\max}} \in \A_{\Trop,\pqs}^{p,q}(\Lambda)$.
We put 
  $$ \int_{\Lambda} \alpha := \sum_{P \in \Lambda_{\max}} m_P \int_{\exp (-P)_\C} \alpha_P .$$
We put 
$$ \langle \alpha, \beta \rangle_{L^2( \Lambda)}
  := \sum_{P \in \Lambda_{\max}} m_P \langle \alpha_P, \beta_P \rangle_{L^{2}(  P  )} 
  =   \sum_{P \in \Lambda_{\max}} m_P \int_{\exp (-P)_\C } (\alpha_P,\beta_P) \omega_{Kah,d} .$$
We also define the \emph{$L^2$-norm} 
  by
  $$\alpha 
    \mapsto \Arrowvert \alpha \Arrowvert_{L^2 ( \Lambda ) }
    := \langle \alpha, \alpha \rangle_{L^2( \Lambda)}^{\frac{1}{2}}.$$
We put 
  $$L_{\Trop}^{p,q,2}(\Lambda):= \bigoplus_{P\in \Lambda_{\max}} L_{\Trop}^{p,q,2}( \relint ( \exp ( - P)_\C  ) ).$$
  This is the completion of $\A_{\Trop,\qs}^{p,q}(\Lambda)$ with respect to the $L^2$-norm $\Arrowvert \alpha \Arrowvert_{L^2 ( \Lambda ) }$.
Similarly, we define 
 $$W_{\Trop}^{p,q,s,2}(\Lambda):= \bigoplus_{P \in \Lambda_{\max} } W_{\Trop}^{p,q,s,2}( \relint ( \exp ( - P)_\C  ) ).$$
and the Sobolev $s$-norm 
  $$\alpha 
    \mapsto \Arrowvert \alpha \Arrowvert_{W^{s,2} ( \Lambda ) }
    := \big(\sum_{P \in \Lambda_{\max}} m_P \Arrowvert \alpha_P \Arrowvert_{W^{s,2}(  P)}^2 \big)^{\frac{1}{2} }  .$$
We also define the $\overline{*}$-operator, 
differential $\overline{\partial}$, codifferential $\overline{\delta}$, and Laplacian $\Delta''$ on $\Lambda$
as the formal sums of them.

\section{Boundary conditions}
In the rest of this paper, 
we fix 
\begin{itemize}
  \item a $d$-dimensional tropically compact (projective) tropical variety $(\Lambda, (m_P)_{P \in \Lambda_{\max}}) $ in $\Trop (T_{\Sigma})$
  and
  \item  an $\R$-K\"{a}hler superform $\omega_{Kah}$ on $\Lambda$.
\end{itemize}

In this section, 
we shall introduce 
normal fields and $4$ boundary conditions $\t''_{\min}$, $\n''_{\min}$, $\t''_{\max}$, and $\n''_{\max}$ for differential forms.
We shall give explicit expressions of $\n''_{\max}$. It will be used in proof of \cref{expression of Sobolev 1 norm}, which shows Gaffney-G{\aa}rding's inequality (\cref{Gaffney-Garding's inequality}).
We also define a new tropical Dolbeault cohomology $H_{\Trop,\Dol,\pqs,\t''_{\min}}^{p,q}( \Lambda)$,
which is isomorphic to the usual tropical Dolbeault cohomology $H_{\Trop,\Dol}^{p,q}( \Lambda)$ when $\Lambda$ is smooth.

\subsection{Normal fields}
We shall define normal fields.

Let $P \in \Lambda_{\max} $ and  $Q \in \Lambda_{\max - 1 }^{\ess}$  a face of $P$.
We fix 
  a basis $x_1,\dots,x_n$ of $M$ 
such that 
\begin{align*}
  \text{the affine span of } (P \cap \R^n )
 = & \{ x \in \R^n \mid  x_i = a_{P,i} \ ( d+1 \leq i \leq n ) \} , \\
   P  \cap \R^n 
 \subset & \{x \in \R^n \mid x_1 \geq a_{P,Q} \}  , \  \text{and} \\
  Q \cap \R^n 
 = & P \cap \{ x \in \R^n \mid  x_1 = a_{P,Q} \}
\end{align*}
  for some $a_{P,Q} , a_{P,i} \in \R$.
We put 
  $z_1,\dots,z_n$ the basis $x_1,\dots,x_n$ as a holomorphic coordinate of $(\C^{\times})^n$.
Rimind that we put $\exp ( - P)_\C \subset T_{\Sigma}(\C)$ the closure of 
  $$ \{z \in  ( \C^{\times} )^n \mid  z_i = e^{-a_{P,i} } \ ( d+1 \leq i \leq n ) \} 
     \cap  (- \log \lvert \cdot \rvert )^{-1} (P)  .$$
Similarly, we put $\exp ( - Q)_\C \subset T_{\Sigma}(\C)$ the closure of 
  $$ \{z \in  ( \C^{\times} )^n \mid z_1=e^{-a_{P,Q}}, \  z_i = e^{-a_{P,i} } \ ( d+1 \leq i \leq n ) \} 
     \cap  (- \log \lvert \cdot \rvert )^{-1} (Q)  .$$

\begin{dfn}\label{definition of N}
The \emph{(inward pointing) unit normal field} 
$$ \N  := \N_P  \in \Gamma (T_\R \exp (- P_\C ) |_{  \exp ( - P)_\C  \cap ( - \log \lvert \cdot \rvert )^{ - 1 } ( Q )  } ) $$
is a vector field 
such that 
  $$g(\N,\N)=1 , \quad g(v,\N )=0  \quad
  ( v \in \Gamma (T_\R   (\exp ( - P)_\C  \cap ( - \log \lvert \cdot \rvert )^{ - 1 } ( Q )  ) ) ) $$
 and it is of the form
  $$ \sum_{ i = 1 }^{ d } f_i \frac{ \partial }{ \partial ( - \frac{1}{2} \log r_i ) } + g_i \frac{ \partial }{ \partial ( - \frac{1}{2} \theta_i ) }$$
  for some positive function $f_1 >0 $ and
      $\R$-valued functions $f_i$ and $ g_i$,
      where $z_i = r_i e^{i \theta_i} $.
  (Strictly speaking, the vector field $\N$ is the smooth extension of a vector field of this form on 
  $$   \exp ( - P)_\C  \cap ( - \log \lvert \cdot \rvert )^{ - 1 } ( Q )  \cap ( \C^{ \times } )^n . ) $$
\end{dfn}

\begin{lem}\label{normal field N explicitly}
  The unit normal field 
$ \N $ is of the form 
  $$\N =  \sum_{ i = 1 }^{ d } f_i ( - \log \lvert z \rvert ) \frac{ \partial }{ \partial ( - \frac{1}{2} \log r_i ) } $$
  for some positive function $f_1 ( - \log \lvert z \rvert ) >0 $ 
      and $\R$-valued functions $f_i ( - \log \lvert z \rvert )$. 
\end{lem}
\begin{proof}
  This follows from the definition of $\R$-K\"{a}hler superforms.
\end{proof}

\begin{dfn}\label{definition N' N''}
  We put
  \begin{align*}
  \N' &:= \N'_P := \frac{1}{\sqrt{2}} ( \N - i I(\N)) \in  \Gamma (T^{1,0} \exp (- P_\C ) |_{  \exp ( - P)_\C  \cap ( - \log \lvert \cdot \rvert )^{ - 1 } ( Q )  } ) ,  \\
   \N''& := \N''_P := \frac{1}{\sqrt{2}} ( \N + i I(\N)) \in  \Gamma (T^{0,1} \exp (- P_\C ) |_{  \exp ( - P)_\C  \cap ( - \log \lvert \cdot \rvert )^{ - 1 } ( Q )  } )  ,
  \end{align*}
   where $I$ is the complex structure.
   We call $\N'$ the \emph{(inward pointing) unit normal $(1,0)$-field} 
   and $\N''$ the \emph{(inward pointing) unit normal $(0,1)$-field}.
   Since $g$ is $I$-invariant, we have 
  $$g(\N',\N'')=1 , \quad g(v,\N' ) = g( v , \N'' ) = 0  \quad
  ( v \in \Gamma (T_\R   (\exp ( - P)_\C  \cap ( - \log \lvert \cdot \rvert )^{ - 1 } ( Q )  )  ) ).$$
\end{dfn}

\begin{dfn}
We call a vector $(1,0)$-field  
a \emph{quasi-smooth $(1,0)$-superfield}
if it is of the form 
  $$ \sum_{i} f_{i}( - \log \lvert z \rvert ) \frac{\partial }{\partial (- \frac{1}{2} \log z_i)}$$
  for some $\R$-valued functions $f_{i}$
\end{dfn}

  The  unit normal $(1,0)$-field $\N'$ is a quasi-smooth $(1,0)$-superfield.

\begin{dfn}
  Let $U \subset \exp ( - P)_\C  $ be an open subset.
  A \emph{local $g$-orthonormal $(1,0)$-frame} on $U$
   is a tuple $(E_i')_{i = 0}^{ d-1 }$ of vector fields $E_i' \in \Gamma ( T^{1,0} U )$
  such that 
   $$g(E_i', E_j'') = \delta_{i,j} $$
  on $ U,$
  where we put $E_i'' := \overline{E_i'} \in \Gamma ( T^{0,1} U )$.
  We call it a \emph{local $g$-orthonormal normal $(1,0)$-frame} (for $Q \subset P$)
  when 
   $$ E_0'|_{U \cap (- \log \lvert \cdot \rvert )^{-1} (Q) } = \N'.$$
   We call it a
\emph{local $g$-orthonormal normal quasi-smooth $(1,0)$-superframe} (for $Q \subset P$) 
if moreover each $E_i'$ is a quasi-smooth $(1,0)$-superfield.
\end{dfn}
A local $g$-orthonormal normal quasi-smooth $(1,0)$-superframe corresponds to differential $(1,0)$-forms $\varphi_i'$ 
(i.e., $\varphi_i' (v) = g (E_i'', v)$ for vector fields $v$)
such that 
  $\omega_{Kah} = i \sum_{i =0}^{d-1} \varphi_i' \wedge \varphi_i''$
  on $U$
and 
  $i \varphi_0', \dots, i \varphi_{d-1}'$ are quasi-smooth superforms.

\begin{rem}
A local $g$-orthonormal normal quasi-smooth $(1,0)$-superframe 
  $E_{P,Q,0}' , \dots, $ $E_{P,Q,d-1}' $  
for $Q \subset P$
exists on $\exp ( -P_\C ) \cap ( \C^{\times} )^n$
(see \cref{quasi-smooth superform orthonormal basis}).
Moreover, by the Gram-Schmids process, we can take 
  $E_{P,Q,0}' , \dots, E_{P,Q,d-1}' \ ( P \in \Lambda_{\max}, \ Q \subset P)$  
so that 
  for $Q \subset P $ and $Q \subset P'$,
  we have 
  $$E_{P,Q,i}'|_{\exp ( - Q )_\C \cap (\C^{\times})^n} 
   = E_{P',Q,i}'|_{\exp ( - Q )_\C \cap (\C^{\times})^n}  \quad ( i \geq 1).$$
In this case, we say 
  $E_{P,Q,0}' , \dots, E_{P,Q,d-1}'  \ ( P \in \Lambda_{\max}, \ Q \subset P)$  are \emph{compatible}.
\end{rem}

\subsection{Boundary conditions}
In this subsection, we shall define $4$ boundary conditions $\t''_{\min}$, $\n''_{\min}$, $\t''_{\max}$, and $\n''_{\max}$,
and give explicit expressions of $\n''_{\max}$.
Pairs of condition $(\t''_{\min}, \n''_{\max})$ and  $(\n''_{\min},\t''_{\max})$ are respectively analogs of Dirichlet and Neumann boundary conditions in Hodge theory of compact Riemann manifolds with smooth boundaries, see \cite{Sch95}.
As in this case, 
we can interpret our conditions using the theory of Hilbert complexes. See \cite{BL92} for Hilbert complexes, in particular,  \cite[Section 4]{BL92} for the case of compact Riemann manifolds with smooth boundaries.
See \cite{Wlo87} for general theory of Sobolev spaces.

\begin{rem}\label{review trace operators}
  We remind \emph{trace operators}.
  There is 
    a continuous linear operator (\cite[Theorem 8.7]{Wlo87})
    $$ - |_{\partial \exp ( - P)_\C  } \colon  W^{1,2}( \relint ( \exp ( - P)_\C  ) ) 
          \ni \omega \mapsto \omega|_{\partial \exp ( - P)_\C  } 
        \in 
        W^{\frac{1}{2},2}( \partial \exp ( - P)_\C   ) $$
  called the \emph{trace operator}
  whose 
    restriction to the subset of smooth functions is the usual restriction,
  where 
    $\partial  \exp ( - P)_\C $ is the boundary of $\exp ( - P)_\C $.
  (See \cite[Definition 4.4]{Wlo87} 
    for the definition of \emph{fractional Sobolev spaces} 
    $  W^{\frac{1}{2},2}( \partial \exp ( - P)_\C   ) $.)
    The only important fact on fractional Sobolev spaces used in this paper (more precisely, in \cref{approximation in codimension 1})
    is that 
    there exists 
    a continuous linear operator (\cite[Theorem 8.8]{Wlo87})
    $$ Z \colon  
        W^{\frac{1}{2},2}( \partial \exp ( - P)_\C   )
        \to W^{1,2}( \relint ( \exp ( - P)_\C  ) ) $$
    such that $ - |_{\partial \exp ( - P)_\C  }  \circ Z = \Id$.
  
  We also put 
  \begin{align*}
    &  - |_{\partial \exp ( - P)_\C  } \colon  W^{p,q,1,2}( \relint ( \exp ( - P)_\C  ) )  \\
    \cong  &  \bigoplus_{\substack{I \in \{1,\dots, d \}^p \\ J \in \{1,\dots, d \}^q }}' W^{1,2}( \relint ( \exp ( - P)_\C  ) ) d z_I \wedge d \overline{z_J} 
          \ni  \omega  = \sum'_{\substack{I \in \{1,\dots, d \}^p \\ J \in \{1,\dots, d \}^q }} \omega_{I,J}(z) d z_I \wedge d \overline{z_J} \\
          \mapsto  &
          \omega|_{\partial \exp ( - P)_\C  } := \sum'_{\substack{I \in \{1,\dots, d \}^p \\ J \in \{1,\dots, d \}^q }} \omega_{I,J}(z) |_{\partial \exp ( - P)_\C  } d z_I \wedge d \overline{z_J}
        \in 
      \bigoplus_{\substack{I \in \{1,\dots, d \}^p \\ J \in \{1,\dots, d \}^q }}' W^{\frac{1}{2},2}( \partial \exp ( - P)_\C  ) d z_I \wedge d \overline{z_J}   
  \end{align*}
  (where $z_1, \dots,z_d \in M$ are as in definition of $\exp (-P)_\C$).
  We also call it the \emph{trace operator}.
\end{rem}

We put 
  $$  T^{0,1}( \exp ( - P)_\C   \cap (- \log \lvert \cdot \rvert )^{-1} ( Q ) )
    \subset 
      T^{0,1} \exp ( - P)_\C 
       |_{ \exp ( - P)_\C   \cap (- \log \lvert \cdot \rvert )^{-1} ( Q ) } $$
the subset of vector fields $v$ such that $g(\N', v )=0$.

\begin{dfn}\label{definition of boundary conditions}
  We say that $\omega = (\omega_P)_{P \in \Lambda_{\max} } \in W_{\Trop}^{p,q,1,2}(\Lambda)$ satisfies conditions $\t''_{\min}, \n''_{\min}$, $ \t''_{\max}, \n''_{\max}$
  if the following conditions hold, respectively.
  \begin{itemize}
    \item ($\t''_{\min}$) 
           For any $Q \in \Lambda_{\max -1}^{\ess}$, 
           there exists 
            $$\eta_Q = \sum'_{I,J} \eta_{Q,I,J} ( - \log \lvert z \rvert ) d ( - \frac{i}{2} \log z_I) \wedge d ( - \frac{1}{2} \log \overline{z_J}) $$
           (where $\eta_{Q,I,J} $ is a Lebesgue measurable function on $Q$)
           such that 
           \begin{align*}
            & \omega_P|_{\wedge^p \Gamma (T^{1,0} \exp ( - P)_\C  |_{\exp ( - P)_\C   \cap (- \log \lvert \cdot \rvert )^{-1} ( Q ) })
              \wedge \wedge^q \Gamma (T^{0,1}( \exp ( - P)_\C   \cap (- \log \lvert \cdot \rvert )^{-1} ( Q ) ) )}  \\
            = &  \eta_Q |_{\wedge^p \Gamma (T^{1,0} \exp ( - P)_\C  |_{  \exp ( - P)_\C   \cap (- \log \lvert \cdot \rvert )^{-1} ( Q )  }) 
              \wedge \wedge^q \Gamma ( T^{0,1}( \exp ( - P)_\C   \cap (- \log \lvert \cdot \rvert )^{-1} ( Q ) ) )} 
           \end{align*}
            for any $P \in \Lambda_{\max}$ containing $Q$.
    \item ($\n''_{\min}$) $\overline{*} \omega $ satisfies condition $\t''_{\min}$.
    \item ($\t''_{\max}$) 
           For any $Q \in \Lambda_{\max -1}^{\ess}$ 
           and 
            $\zeta \in\A_{\Trop,\qs}^{d-p,d-q-1} ( \Lambda ),$
           we have 
            $$\sum_{\substack{P \in \Lambda_{\max} \\  Q \subset P } } 
             m_P \int_{  \exp ( - P)_\C  \cap ( - \log \lvert \cdot \rvert )^{ - 1 } ( Q )   }
            \omega_P \wedge   \zeta =0 .$$
    \item ($\n''_{\max}$) $\overline{*} \omega $ satisfies condition $\t''_{\max}$.
  \end{itemize}
\end{dfn}

  (Here we mean that condition $\t''_{\min}$ and $\t''_{\max}$ are empty condition when $q= d$, 
  and condition $\n''_{\min}$ and $\n''_{\max}$ are empty condition when $q= 0$.)

\begin{rem}
Let $\omega \in W_{\Trop}^{p,q,1,2}(\Lambda)$.
 \begin{itemize}
  \item When $\omega$ is a quasi-smooth superform, it satisfies condition $\t''_{\min}$.
  \item When $\omega$ satisfies condition $\t''_{\min}$ (resp. $\n''_{\min} $), it also satisfies condition $\t''_{\max}$ (resp. $\n''_{\max} $). This follows from 
  the fact 
  (which follows from the boundary condition (\cref{definition of tropical variety}) of tropical varieties)
  that
    for $Q \in \Lambda_{\max -1}^{\ess}$ 
         and 
        $ f  \in \A_{\Trop,\qs}^{d,d-1} ( \Lambda ),$
    we have 
      $$\sum_{ \substack{P \in \Lambda_{\max} \\ Q \subset P } } 
         m_P 
          \int_{ \exp ( - P)_\C  \cap ( - \log \lvert \cdot \rvert )^{ - 1 } ( Q )  }    f
         =0 .$$
  (See \cite[Proposition 3.8]{Gub16}, \cite[Theorem 4.9]{JSS19}.)
 \end{itemize}
\end{rem}

\begin{rem}
  A differential form
$\omega = (\omega_P)_{P \in \Lambda_{\max} } \in W_{\Trop}^{p,q,1,2}(\Lambda)$
satisfies condition $\t''_{\min}$
if and only if  for each $Q \in \Lambda_{\max -1}^{\ess}$,
$$v_P' \in \Gamma (T^{1,0} \exp (-P)_\C |_{\exp (-Q)_\C}),
 \ v_i' \in \Gamma (T^{1,0} \exp (-Q)_\C),
 \ w_j' \in \Gamma (T^{0,1} \exp (-Q)_\C)$$
such that 
$\sum_{\substack{ P \in \Lambda_{\max} \\ Q \subset P} }
v_P' =0$, we have
$$\sum_{\substack{ P \in \Lambda_{\max} \\ Q \subset P} }
\omega_P (v_P', v_1',\dots, v_{p-1}',w_1'', \dots, w_q'') =0$$
on $\exp ( -Q)_\C \cap (\C^{\times})^n$.
\end{rem}

\begin{lem}\label{boundary conditions in Q-smooth in codimension 1 case}
  When $\Lambda$ is $\Q$-smooth in codimension $1$ (\cref{definition of locally irreducible tropical varieties}),
  equalities of conditions 
  $\t''_{\min} = \t''_{\max}$
  and 
  $\n''_{\min} = \n''_{\max}$ hold.
\end{lem}
\begin{proof}
  Straightforward. 
\end{proof}

\begin{rem}
Gubler-Jell-Rabinoff introduced another kind of new superforms, \emph{weakly smooth forms}(\cite{GJR21-1}, \cite{GJR21-2}).
  They are related to conditions $\t''_{\max}$ and $\n''_{\max}$.
\end{rem}

\begin{lem}\label{n max explicitly}
  A differential form $\omega = (\omega_P)_{P \in \Lambda_{\max}} \in W_{\Trop}^{p,q,1,2}(\Lambda)$
  satisfies condition $\n''_{\max}$
if and only if 
  for $Q \in \Lambda_{\max -1}^{\ess} $,
compatible local $g$-orthonormal normal quasi-smooth $(1,0)$-superframes
  $E_{P,Q,0}' , \dots, E_{P,Q,d-1}'  \ ( P \in \Lambda_{\max}, \ Q \subset P)$ on $\exp (-P)_\C \cap (\C^{\times})^n$, 
   a quasi-smooth superform $ \zeta \in \A_{\Trop,\qs}^{p,q-1} ( \Lambda ) $,
    and
      $I \in  \{0,1,\dots,d-1 \}^p $  and
      $J \in  \{1,\dots,d-1 \}^{q-1} $ 
  we have
  $$ \sum_{\substack{P \in \Lambda_{\max} \\ Q \subset P} }
    \frac{m_P}{f_{P,1} ( -\log \lvert z \rvert ) }
    \zeta (E_{P,Q,I}', E_{P,Q,J}'') \overline{\omega_P ( \N''_P , E_{P,Q,I}', E_{P,Q,J}'' )}  = 0  $$
  on $ \exp ( - Q )_\C  \cap (\C^{\times})^n$.
\end{lem}
\begin{proof}
  We fix $Q \in \Lambda_{\max -1 }^{\ess}$.
  We assume that $\omega$ satisfies condition $\n''_{\max}$.
  Then 
  we have 
  $$\sum_{ \substack{P \in \Lambda_{\max} \\ Q \subset P}}
     m_P \zeta \wedge \overline{*}\omega_P (\frac{\N'_P}{f_{P,1}(- \log \lvert z \rvert ) }, E'_{P,Q, (1,\dots,d-1)}, E''_{P,Q, (1,\dots,d-1)} )  
    = 0. $$
  We may assume that 
$\zeta (E_{P,Q,I}', E_{P,Q,J}'')$
($P \in \Lambda_{\max}, \ Q \subset P$) are possibly non-zero only for a pair $(I,J)$,
hence the assertion holds.
(To see this, 
let  $\varphi_{P,Q,i}'$ be differential forms corrsponding to quasi-smooth $(1,0)$-superfields extending $E_{P,Q,i}'$ to a neighborhood in $(\C^{\times})^n$.
Then without loss of generality, we may assume that $\zeta$ is a $C_{\Trop,\qs}^{\infty}$-linear sum of products of $i\varphi_{P,Q,i}'$ and $\varphi_{P,Q,j}''$.)
\end{proof}

\begin{lem}\label{n max for strange superframe}
  There exists 
  a local $g$-orthonormal normal $(1,0)$-frame 
  $E_{P,Q,0}' , \dots, E_{P,Q,d-1}' $ 
  on a neighborhood $U_{P,Q}$ of $\exp ( - P)_\C  \cap ( - \log \lvert \cdot \rvert )^{-1} (Q)$ in $\exp ( - P)_\C$
  for each pair $ Q \subset P \ (P \in \Lambda_{\max}, \ Q \in \Lambda_{\max - 1 }^{\ess})  $
  such that 
  for $Q \subset P $ and $Q \subset P'$,
  we have 
  $$E_{P,Q,i}'|_{\exp ( - Q )_\C } 
   = E_{P',Q,i}'|_{\exp ( - Q )_\C }  \quad ( i \geq 1)$$
  and
    $$E_{P,Q,0}',   e^{ i \sum_j \alpha_{1,j} \theta_j } E_{P,Q,1}', \dots, e^{ i \sum_j \alpha_{d-1,j} \theta_j } E_{P,Q,d-1}' $$
    is a local $g$-orthonormal normal quasi-smooth $(1,0)$-superframe 
   on $U_{P,Q} \cap (\C^{\times})^n$ 
   for some $\alpha_{i,j} \in \Z$,
     where $z_1= r_1 e^{i \theta_1} , \dots, z_n = r_n e^{i \theta_n} \in M$ is a basis as a holomorphic coordinate of $(\C^{\times})^n$.

  In particular, 
  for 
    a differential form $\omega = (\omega_P)_{P \in \Lambda_{\max}} \in W_{\Trop}^{p,q,1,2}(\Lambda)$ 
    satisfying condition $\n''_{\max}$,
   a polyhedron $Q \in \Lambda_{\max -1}^{\ess} $,
   a quasi-smooth superform $ \zeta = (\zeta_P) \in \A_{\Trop,\qs}^{p,q-1} ( \Lambda ) $,
    and
      $I \in \{0,\dots,d-1 \}^p $, 
      $J \in \{1,\dots,d-1 \}^{q-1} $ 
  we have
  $$ \sum_{\substack{P \in \Lambda_{\max} \\ Q \subset P} }
    \frac{m_P}{f_{P,1} ( -\log \lvert z \rvert ) }
    \zeta_P (E_{P,Q,I}', E_{P,Q,J}'') \overline{\omega_P ( \N''_P , E_{P,Q,I}', E_{P,Q,J}'' )}  = 0  $$
  on $ \exp ( - Q )_\C$.
\end{lem}
\begin{proof}
  The first assertion follows from the Gram-Schmidt process.
  The second assertion follows from \cref{n max explicitly}.
\end{proof}

\begin{lem}\label{t''wedge * n'' restriction }
  Let $P \in \Lambda_{\max}$ be a polyhedron and $Q \in \Lambda_{\max -1 }^{\ess}$ contained in $P$.
  Let $\omega \in W^{p,q,1,2}( \relint (\exp ( - P)_\C )  )$
     and $\eta \in W^{p,q+1, 1,2 }( \relint ( \exp ( - P)_\C )  )$.
  We put $\chi :=  \omega \wedge \overline{*}  \eta $.
  Then 
    $$ \chi 
    |_{ \wedge^{2 d-1} \Gamma ( T ( \exp ( - P)_\C  \cap ( - \log \lvert \cdot \rvert )^{-1} ( Q)) ) }
    = \frac{1}{ \sqrt{2}} (\omega, \eta (\N'') )\mu_{ \partial \exp ( - P)_\C  } 
    |_{ \exp ( - P)_\C  \cap ( - \log \lvert \cdot \rvert )^{-1} ( Q))  } ,$$
  where $\mu_{ \partial \exp ( - P)_\C  } $ is the volume form on $ \partial \exp ( - P)_\C $ induced by the volume form $\omega_{Kah,d}$ on $\exp (-P)_\C$.
\end{lem}
Here $(\omega, \eta (\N''))$ means 
  the restriction of $(\omega, \tilde{\eta (\N'')})$ 
  to $ \exp ( - P)_\C  \cap ( - \log \lvert \cdot \rvert )^{-1} ( Q)) $
  for any extension $\tilde{\eta (\N'')} $ to $\exp ( - P)_\C $ 
    of $\eta(\N'') $.

\begin{proof}
  Proof is parallel to \cite[Proposition 1.2.6 (c)]{Sch95}.
  Let $E_i'$ be a local $g$-orthonomal normal $(1,0)$-frame on $\exp (-P)_\C \cap (\C^{\times})^n$.
  We put $S (p,d)$ (resp. $S(q,d-1)$) the set of permutations $\sigma_1$ of $\{0,\dots, d-1\}$
  (resp. $\sigma_2$ of $\{1,\dots,d-1\}$)
  such that 
    $\sigma_1 (0) < \dots < \sigma_1 ( p - 1)$ and 
    $\sigma_1 (p) < \dots < \sigma_1 ( d - 1)$
    (resp. $\sigma_2 (1) < \dots < \sigma_2 ( q )$ and 
    $\sigma_2 (q + 1) < \dots < \sigma_2 ( d - 1)$).
    Then on $\exp ( - P)_\C  \cap ( - \log \lvert \cdot \rvert )^{-1} ( Q)$, we have 
  \begin{align*}
    &  \chi (I(\N), E_1',E_1'', \dots, E_{d-1}', E_{d-1}'')  \\
    = &  \frac{i}{\sqrt{2}} (-1)^{\frac{(d-1)(d-2)}{2}} \chi (\N', E_1', \dots,E_{d-1}',  E_1'', \dots,  E_{d-1}'')  \\
    = & \frac{1}{\sqrt{2}} i^{d^2 + d -1 } 
    \sum_{ \substack{ \sigma_1 \in S(p,d) \\ \sigma_2 \in S(q,d-1) } }
    \sign (\sigma_1) \sign ( \sigma_2 ) 
    ( -1 )^{q (d-p)}
     \omega (E_{\sigma_1 (0)}', \dots,E_{\sigma_1(p-1)}',  E_{\sigma_2(1)}'', \dots,  E_{\sigma_2 ( q ) }'')   \\
   & \qquad \qquad \qquad \cdot \overline{*} \eta (E_{\sigma_1 (p)}', \dots,E_{\sigma_1 ( d -1 ) }',  E_{\sigma_2( q + 1 ) }'', \dots,  E_{\sigma_2 ( d-1 ) }'')   \\
    = &  \frac{1}{\sqrt{2}} i^{d-1}
    \sum_{ \substack{ \sigma_1 \in S(p,d) \\ \sigma_2 \in S(q,d-1) } }
    ( -1 )^{p}
     \omega (E_{\sigma_1 (0)}', \dots,E_{\sigma_1(p-1)}',  E_{\sigma_2(1)}'', \dots,  E_{\sigma_2 ( q ) }'')   \\
     & \qquad \qquad \qquad \cdot \overline{ \eta (E_{\sigma_1 (0)}', \dots,E_{\sigma_1(p-1)}', \N'', E_{\sigma_2(1)}'', \dots,  E_{\sigma_2 ( q ) }'')  } \\
    = & \frac{1}{\sqrt{2}} i^{d-1}
     (\omega , \eta ( \N'') )\\ 
    = & \frac{1}{\sqrt{2}} 
     ( \omega , \eta ( \N'') )
      \mu_{\partial \exp ( - P)_\C  } (I(\N), E_1',E_1'', \dots, E_{d-1}', E_{d-1}'') .
  \end{align*}
\end{proof}

\subsection{Another Dolbeault cohomology}
We shall define another tropical Dolbeault cohomology using condition $\t''_{\min}$.

\begin{lem}\label{condition t min max preserved by partial}
  Let $\omega \in W_{\Trop}^{p,q,2,2} ( \Lambda )$ satisfies condition $\t''_{\min}$ or $\t''_{\max}$.
  Then $\overline{\partial}  \omega$ satisfies the same condition.
\end{lem}
\begin{proof}
  The case of condition $\t''_{\min}$ is trivial.
  We assume that $\omega$ satisfies condition $\t''_{\max}$.
  It suffices to show that 
  for $Q \in \Lambda_{\max - 1 }^{\ess}$ 
      and 
      $\xi \in \A_{\Trop,\qs}^{ d-p , d-q -2 } ( \Lambda ) $ 
      such that $\xi |_{ (-\log \lvert \cdot \rvert )^{-1} ( Q' ) }=0$ for $Q' \in \Lambda_{\max - 1 }^{\ess} \setminus \{Q\}$,
  we have 
    $$\sum_{\substack{P \in \Lambda_{\max} \\ Q \subset P }} 
             m_P \int_{  \exp ( - P)_\C  \cap ( - \log \lvert \cdot \rvert )^{ - 1 } ( Q )   }
            \overline{\partial}\omega_P \wedge \xi =0 .$$
  By Stokes' theorem and condition $\t''_{\max}$ for $\omega$, 
  we have 
  \begin{align*}
    & \sum_{\substack{P \in \Lambda_{\max} \\ Q \subset P }} 
             m_P \int_{  \exp ( - P)_\C  \cap ( - \log \lvert \cdot \rvert )^{ - 1 } ( Q )   }
            \overline{\partial}\omega_P \wedge \xi \\
     = & \sum_{\substack{P \in \Lambda_{\max} \\ Q \subset P }} 
             m_P \int_{  \exp ( - P)_\C  \cap ( - \log \lvert \cdot \rvert )^{ - 1 } ( Q )   }
            \overline{\partial}(\omega_P \wedge \xi) 
        -    \omega_P \wedge \overline{\partial}\xi  \\
    = & 0.
  \end{align*}
\end{proof}

\begin{cor}\label{condition n min max preserved by delta}
  Let $\eta \in W_{\Trop}^{p,q,2,2} ( \Lambda )$ satisfies condition $\n''_{\min}$ or $\n''_{\max}$.
  Then $\overline{\delta}  \eta$ satisfies the same condition.
\end{cor}

  For $\epsilon, \epsilon_1, \epsilon_2 \in \{ \t''_{\min} , \t''_{\max} , \n''_{\min} , \n''_{\max} \} $, we put 
   $$\A_{\Trop, \pqs, \epsilon}^{p,q}(\Lambda) \quad (\text{resp.} \ \A_{\Trop, \pqs, \epsilon_1, \epsilon_2}^{p,q}(\Lambda) ) $$
  the subspace of 
  $ \A_{\Trop, \pqs}^{p,q}(\Lambda)$ consisting of elements satisfying the condition $\epsilon$ (resp. the conditions $\epsilon_1$ and $\epsilon_2$).

\begin{prp}\label{exactness of complex of pqs t min superforms}
  For any $p$,
  the complex $(\A_{\Trop, \pqs, \t''_{\min} }^{p,*}, \overline{\partial})$ is a fine resolution of 
  $\Ker (\overline{\partial} \colon \A_{\Trop, \pqs, \t''_{\min} }^{p,0} \to \A_{\Trop, \pqs, \t''_{\min} }^{p,1})$ on $X$.
\end{prp}
\begin{proof}
  Similar to \cref{exactness of complex of quasi-smooth forms}. We omit details.
  The condition $\t''_{\min}$ is kept under our version of $G^*$ (used in proof of \cref{exactness of complex of quasi-smooth forms}) in \cite[Section 2.2]{Jel16A}. 
\end{proof}

We put 
$$ H_{\Trop,\Dol, \pqs, \t''_{\min} }^{p,q}( \Lambda )
  := H^q ( ( \mathscr{A}_{\Trop, \pqs, \t''_{\min} }^{p,*}(\Lambda) , \overline{\partial}) ) $$
By Poincar\'{e}'s lemma (\cref{exactness of complex of pqs t min superforms}), 
this cohomology group
  $ H_{\Trop,\Dol, \pqs, \t''_{\min} }^{p,q}( \Lambda ) $
is isomorphic to the $q$-th sheaf cohomology group
  $$ H^q ( X, 
      \Ker (\mathscr{A}_{\Trop, \pqs, \t''_{\min} }^{p,0} 
            \to \mathscr{A}_{\Trop, \pqs, \t''_{\min} }^{p,1} ) ) .$$
A natural map
$$  \Ker (\mathscr{A}_{\Trop, \qs }^{p,0} 
         \to \mathscr{A}_{\Trop, \qs }^{p,1} ) ) 
  \to 
    \Ker (\mathscr{A}_{\Trop, \pqs, \t''_{\min} }^{p,0} 
         \to \mathscr{A}_{\Trop, \pqs, \t''_{\min} }^{p,1} ) ) $$
of sheaves on $X$
is an isomorphism
if and only if 
for any $R \in \Lambda^{\ess}$, 
the star fan 
$$  \Star (R) := \bigcup_{ \substack{P \in \Lambda_{\max} \\  R \subset P} }  \R_{\geq 0} \cdot \overline{P \cap \R^n}  \subset \R^n / \Tan ( R \cap \R^n )$$
(where $\overline{P \cap \R^n} \subset \R^n / \Tan ( R \cap \R^n )$ 
is the image of $P \cap \R^n$ under the projection)
 is uniquely $p$-balanced in the sense of \cite[Definition 4.3.2]{Aks19}.
This condition holds when
for any $R \in \Lambda^{\ess}$, 
the star fan $\Star (R)$
satisfies Poincar\'{e} duality (\cite[Theorem 4.3.5]{Aks19}), e.g., $\Lambda$ is locally matroidal.
See also \cite[Theorem 1.5]{AP21}.

\section{Gaffney-G{\aa}rding's inequality}
In this section, we shall prove Gaffney-G{\aa}rding's inequality (\cref{Gaffney-Garding's inequality}).
Our proof is similar to that (\cite[Corollary 2.1]{Sch95}) of Gaffney's inequality for compact Riemann manifolds with smooth boundaries.
(See \cite[Section 0.6]{GH78} for the case of compact K\"{a}hler manifolds.)

\begin{prp}[Green's formula]\label{Green's formula}
  Let $\omega \in W^{p,q-1,1,2}(\relint(\exp ( - P)_\C ))$
  and $\eta \in W^{p,q,1,2}(\relint( \exp ( - P)_\C  ) )$.
  Then 
  $$ \langle \overline{\partial} \omega  , \eta \rangle 
  =  \langle \omega, \overline{\delta} \eta \rangle 
  + \int_{\partial ( \exp ( - P)_\C  )} \omega \wedge \overline{*}  \eta.$$
\end{prp}
\begin{proof}
  This follows from Stokes' theorem.
\end{proof}

\begin{dfn}
We define \emph{the Dirichlet integral} by
$$\D \colon W_{\Trop}^{p,q,1,2}(\Lambda) \times W_{\Trop}^{p,q,1,2}(\Lambda)\ni (\omega,\eta) 
\mapsto 
\sum_{P \in \Lambda_{\max}} m_P \langle \overline{\partial} \omega_P, \overline{\partial} \eta_P\rangle 
+ m_P \langle \overline{\delta} \omega_P , \overline{\delta} \eta_P \rangle  \in \R.$$
\end{dfn}

\begin{cor}\label{the Dirichlet integral}
  Let $\omega \in W^{p,q,2,2}(\Lambda)$
  and $\eta \in W^{p,q,1,2}(\Lambda)$.
  Then 
  $$\D(\omega,\eta)  = \sum_{P \in \Lambda_{\max}} m_P 
   \bigg(\langle \Delta''_P \omega_P,  \eta_P\rangle 
  +  \overline{\int_{\partial \exp ( - P)_\C  } \eta_P \wedge \overline{*} \overline{\partial} \omega_P}
  - \int_{\partial \exp ( - P)_\C  } \overline{\delta} \omega_P \wedge \overline{*}  \eta_P \bigg) .$$
\end{cor}

\begin{lem}\label{Corollary 2.1.3 (b)}
  Let $\omega \in \A^{0,1}(\exp ( - P)_\C )$.
  Then 
  $$ \int_{\exp ( - P)_\C } (\overline{\delta} \omega) \omega_{Kah,d} 
    = - \frac{1}{\sqrt{2}}  \int_{\partial \exp ( - P)_\C } 
      \omega (\N'')\mu_{\partial \exp ( - P)_\C }.$$
\end{lem}
\begin{proof}
  This follows from Stokes' theorem,
  $(\overline{\delta} \omega) \omega_{Kah,d} 
  = -  \overline{\overline{\partial} \overline{*} \omega  }$, and 
  $ \overline{*} \omega |_{\wedge^{2d-1} \Gamma (T\partial \exp ( - P)_\C) }
    = \frac{1}{\sqrt{2}}
      \overline{\omega (\N'')} \mu_{\partial \exp ( - P)_\C }$.
\end{proof}

\begin{dfn}
  For $s \geq 1$ and $\epsilon, \epsilon_1, \epsilon_2 \in \{ \t''_{\min} , \t''_{\max} , \n''_{\min} , \n''_{\max} \} $, we put 
  $$ W_{\Trop, \epsilon}^{p,q,s,2}(\Lambda)  \quad (\text{resp. } W_{\Trop, \epsilon_1,\epsilon_2}^{p,q,s,2}(\Lambda))$$
  the Hilbert subspace of 
  $ W_{\Trop}^{p,q,s,2}(\Lambda)$ consisting of elements satisfying the condition $\epsilon$ (resp. the conditions $\epsilon_1$ and $\epsilon_2$).
\end{dfn}

We will give a proof of the following in Section 9.
\begin{prp}\label{density H = W with boundary conditions}
  For $\epsilon, \epsilon_1, \epsilon_2 \in \{ \t''_{\min} , \t''_{\max} , \n''_{\min} , \n''_{\max} \} $, 
  $$ \A_{\Trop, \pqs, \epsilon}^{p,q}(\Lambda) \subset W_{\Trop, \epsilon}^{p,q,1,2}(\Lambda)   ,
   \quad \A_{\Trop, \pqs, \epsilon_1,\epsilon_2}^{p,q}(\Lambda) \subset W_{\Trop, \epsilon_1,\epsilon_2}^{p,q,1,2}(\Lambda)$$
  are dense.
\end{prp}

\begin{thm}\label{expression of Sobolev 1 norm}
  There exist 
    smooth bundle endomorphisms 
    $$S_{P,Q} \colon \A^{p,q}( \exp ( - P)_\C  )
        |_{ \exp ( - P)_\C  \cap ( - \log \lvert \cdot \rvert )^{-1} ( Q ) } 
       \to
        \A^{p,q}( \exp ( - P)_\C  )
        |_{ \exp ( - P)_\C  \cap ( - \log \lvert \cdot \rvert )^{-1} ( Q ) } $$
    ($ P \in \Lambda_{\max}$ and $Q \in \Lambda_{\max - 1 }^{\ess}$
    contained in $P$)
  such that 
   for $\omega = ( \omega_P )_{P \in \Lambda_{\max} }\in W_{\Trop, \t''_{\min},\n''_{\max}}^{p,q,1,2}(\Lambda)$, 
  we have 
  \begin{align*}
    &\Arrowvert \omega\Arrowvert_{W^{1,2} ( \Lambda ) }^2  \\
    = & \lVert \omega \rVert_{L^2 ( \Lambda ) }^2
       + \Re \langle \mathcal{R}^W \omega, \omega \rangle 
       + 2 \D(\omega,\omega) \\
    + & \sum_{P \in \Lambda_{\max}} \sum_{ \substack{Q \in \Lambda_{\max -1}^{\ess}  \\ Q \subset P } }
        m_P \int_{ \exp ( - P)_\C  \cap ( - \log \lvert \cdot \rvert )^{-1} ( Q) } 
        \Re ( S_{P,Q} \omega_P,\omega_P )
         \mu_{\partial \exp ( - P)_\C  },
  \end{align*}
  where 
  $\Re$ means the real part, 
  and
  $\mathcal{R}^W $ is a smooth bundle morphism defined by \cite[Equation 2.24 in Chapter 1]{Sch95} for a fixed $g$-orthonomal frame $E_1,\dots,E_{2d}$ on $\exp ( - P)_\C$ (the ``Ricci operator'' associated to $\mathcal{R}^{\Lambda}$).
\end{thm}

\begin{rem}\label{Weizenbock formula}
By Weizenb\"{o}ck formula (\cite[Equation 2.23 in Chapter 1]{Sch95}), we have 
$$\sum_{i =1}^{2d}
      ( \nabla_{E_i} \omega_P,\nabla_{E_i} \omega_P ) 
      = - \Delta'' ( \omega_P,  \omega_P )
        + 2 \Re  ( \Delta'' \omega_P,  \omega_P )
        + \Re  ( \mathcal{R}^W \omega_P,  \omega_P ) ,$$
  where $\nabla $ is the Levi-Civita connection 
(cf. \cite[Equation 1.8 in Chapter 2]{Sch95}).
\end{rem}

\begin{proof}
  By \cref{density H = W with boundary conditions},
    we may and do assume  $\omega \in \A_{\Trop,\pqs, \t''_{\min}, \n''_{\max} }^{p,q}(\Lambda)$.
 By definition, 
 we have 
   $$\lVert \omega\rVert_{W^{1,2} ( \Lambda ) }^2 
   = \lVert \omega \rVert_{L^2 ( \Lambda ) }^2
   + \sum_{P \in \Lambda_{\max}} \sum_{i =1 }^{ 2d } 
     m_P \int_{ \exp ( - P)_\C } ( \nabla_{E_i} \omega_P,\nabla_{E_i} \omega_P ) \omega_{Kah,d} .$$
Hence by \cref{Weizenbock formula}, it suffices to show that there exist 
    smooth bundle endomorphisms
    $S_{P,Q} $
such that 
\begin{align*}
 & \sum_{P \in \Lambda_{\max}} \sum_{\substack{Q \in \Lambda_{\max -1}^{\ess}  \\ Q \subset P }  }
  m_P \int_{ \exp ( - P)_\C  \cap ( - \log \lvert \cdot \rvert )^{-1} ( Q) } 
    \Re ( S_{P,Q} \omega_P,\omega_P  )
  \mu_{\partial \exp ( - P)_\C  }  \\
 = & \sum_{P \in \Lambda_{\max}}   
  m_P \int_{ \exp ( - P)_\C  } 
     (2 \Re ( \Delta'' \omega_P,\omega_P  )
   - \Delta'' ( \omega_P,\omega_P ) )
  \omega_{Kah,d}
  - 2 \D(\omega,\omega) .
\end{align*}
By \cref{the Dirichlet integral} and \cref{Corollary 2.1.3 (b)},
it suffices to show that there exist 
 smooth bundle endomorphism $S_{P,Q}$ 
 such that for $Q \in \Lambda_{\max -1 }^{\ess}$,
 \begin{align*}
 & \sum_{ \substack{P \in \Lambda_{\max} \\ Q \subset P }} 
  m_P  \int_{ \exp ( - P)_\C  \cap ( - \log \lvert \cdot \rvert )^{-1} ( Q) } 
     \Re (S_{P,Q} \omega_P,\omega_P )
  \mu_{\partial \exp ( - P)_\C  }  \\
 = & 
    \sum_{\substack{P \in \Lambda_{\max} \\ Q \subset P }}   
  m_P \int_{ \exp ( - P)_\C  \cap ( - \log \lvert \cdot \rvert )^{-1} ( Q) } 
  - 2 \Re ( \omega_P \wedge \overline{*} \overline{\partial} \omega_P)
    + 2 \Re (\overline{\delta} \omega_P \wedge \overline{*} \omega_P ) \\
  & + \frac{1}{\sqrt{2}} D((\omega_P,\omega_P))(\N_P'')  \mu_{ \partial \exp ( - P)_\C  }  ,
 \end{align*}
 where $D((\omega_P,\omega_P))(\N_P'')$ is the derivative of $(\omega_P,\omega_P)$ in the direction of $\N_P''$.
  By \cref{t''wedge * n'' restriction },
  the right hand side equals 
  \begin{align*}
   &\sum_{\substack{P \in \Lambda_{\max} \\ Q \subset P }}   
     m_P \int_{ \exp ( - P)_\C  \cap ( - \log \lvert \cdot \rvert )^{-1} ( Q) } 
     (   -  \sqrt{2} \Re ( \omega_P, \overline{ \partial} \omega_P ( \N''_P ) )
     + \sqrt{2} \Re ( \overline{ \delta} \omega_P,  \omega_P ( \N''_P ) )  \\
    & + \frac{1}{\sqrt{2}} D((\omega_P,\omega_P))(\N_P'')  )
     \mu_{ \partial \exp ( - P)_\C  }   .
  \end{align*}

  We fix $Q \in \Lambda_{\max -1 }^{\ess}$.
  Let $E_{P,Q,0}',\dots,E_{P,Q,d-1}'$ be a local $g$-orthonormal normal $(1,0)$-frame as in \cref{n max for strange superframe} for each $P $ containing $Q$. 
  For simplicity, we put $E_i' := E_{P,Q,i}'$ when there is no confusion.
  \begin{align*}
    & ( \omega_P, \overline{ \partial} \omega_P ( \N''_P ) ) \\
    = & \sum'_{\substack{I,J \\ 0 \notin  J}}
         \omega_P(E_I',E_J'') 
         \cdot \overline{\overline{\partial} \omega_P (\N''_P , E_I' , E_J'' ) }  \\
    = & \sum'_{\substack{I,J \\ 0 \notin   J}}
         \omega_P(E_I',E_J'') 
         \cdot \bigg(
         \sum_{1 \leq k \leq q} (-1)^{ p + k }
           \overline{(\nabla_{E_{j_k}''} \omega_P) (\N''_P , E_I' , E_{J \setminus \{j_k\} }'' ) }  
         + \overline{(\nabla_{\N_P''} \omega_P) ( E_I' , E_J'' ) }  
         \bigg).
  \end{align*}
    Hence we have
  \begin{align*}
    &  ( \omega_P, \overline{ \partial} \omega_P ( \N''_P ) )
    - ( \omega_P, S_{P,Q,1} \omega_P )  \\
   = &   \sum'_{\substack{I,J \\ 0 \notin J}}
         \omega_P(E_I',E_J'') 
         \cdot \bigg(
         \sum_{1 \leq k \leq q} (-1)^{ p + k }
           \overline{ D (\omega_P (\N''_P , E_I' , E_{J \setminus \{j_k\} }'' ) )  (E_{j_k}'')}
         + \overline{ D( \omega_P ( E_I' , E_J'' )) ( \N''_P ) }  
         \bigg)  
  \end{align*}
  for some smooth endomorphism $S_{P,Q,1}$.
  Similarly,
  \begin{align*}
    & ( \overline{ \delta} \omega_P,  \omega_P ( \N''_P ) )
    -  (S_{P,Q,2} \omega_P,\omega_P )  \\
   = - &  \sum'_{\substack{I,J \\ 0 \notin  J}}
          \bigg(
         \sum_{ j \notin J} 
           D ( \omega_P (E_j'', E_I' , E_J'' ) ) (E_j')  
         + D(\omega_P ( \N''_P,  E_I' , E_J'' ) ) (\N'_P)  
         \bigg) 
         \cdot \overline{\omega_P( \N''_P, E_I',E_J'') }
  \end{align*}
  for some smooth endomorphism $S_{P,Q,2}$.
 Since  
$$  \frac{1}{\sqrt{2}} D((\omega_P,\omega_P))(\N_P'')   
   = \frac{1}{2} D((\omega_P,\omega_P))(\N_P)   ,$$
  we have
  \begin{align*}
   & -  \sqrt{2} \Re \bigg(  \sum'_{\substack{I,J \\ 0 \notin J}}
         \omega_P(E_I',E_J'') 
         \cdot 
          \overline{ D( \omega_P ( E_I' , E_J'' ) ) ( \N''_P ) } 
       \bigg) \\
    & - \sqrt{2} \Re \bigg( \sum'_{\substack{I,J \\ 0 \notin  J}}
          D(\omega_P ( \N''_P,  E_I' , E_J'' ) ) (\N'_P)  
         \cdot \overline{\omega_P( \N''_P, E_I',E_J'') }
        \bigg)  \\
    & + \frac{1}{\sqrt{2}} D((\omega_P,\omega_P))(\N_P'')   \\
    = &  \frac{i}{2} 
         \sum'_{\substack{I,J \\ 0 \notin  J}}
         \bigg(  \omega_P(E_I',E_J'') 
         \cdot 
           D( \overline{\omega_P ( E_I' , E_J'' )} ) ( I(\N_P) ) 
       - \overline{ \omega_P(E_I',E_J'') }
         \cdot 
          D( \omega_P ( E_I' , E_J'' ) ) ( I(\N_P) )    \\
     &  +    D(\omega_P ( \N''_P,  E_I' , E_J'' ) ) ( I(\N_P) )  
         \cdot \overline{\omega_P( \N''_P, E_I',E_J'') }  \\
     &  - \omega_P( \N''_P, E_I',E_J'') 
         \cdot 
          D(\overline{ \omega_P ( \N''_P, E_I' , E_J'' )} ) ( I (\N_P) )   \bigg).
     \end{align*}
  By \cref{normal field N explicitly} and the condition on $E_i'$ (\cref{n max for strange superframe}), 
  this can be expressed 
  as $\Re (S_{P,Q,3} \omega_P,\omega_P ) $
  for some smooth endomorphism $S_{P,Q,3}$.
  
  Consequently,
  the remaining parts are
  \begin{align*}
   &  \sum_{ \substack{P \in \Lambda_{\max} \\ Q \subset P} }   
     m_P \int_{ \exp ( - P)_\C  \cap ( - \log \lvert \cdot \rvert )^{-1} ( Q) } 
     \Re (D ( \omega_P (E_I' , E_J'' )
        \overline{\omega_P( \N''_P, E_I',E_{J \setminus \{j \} }'') }
       ) (E_j') )
     \mu_{ \partial \exp ( - P)_\C  }  
  \end{align*}    
    ($I \in \{0,1,\dots,d-1\}^p ,J \in \{1,\dots,d-1\}^q$ and $j \in J$).
  Let $\mu_{ \exp ( - Q )_\C} :=  \omega_{Kah,d-1}$ be the volume form on $\exp (-Q)_\C$.
  We fix an identification
    $$\exp ( - P)_\C  \cap ( - \log \lvert \cdot \rvert )^{-1} ( Q) 
      = S^1 \times \exp ( - Q )_\C.$$
  We put 
    $$\mu_{ \partial \exp ( - P)_\C  } = \omega_{Kah,d} (\N)
      = \frac{1}{f_{P,1} ( - \log \lvert z \rvert )} d \theta_{P,1} \wedge \mu_{ \exp ( - Q )_\C}$$
    on $\exp ( - P)_\C  \cap ( - \log \lvert \cdot \rvert )^{-1} ( Q)$,
    where $e^{i \theta_{P,1}}$ is the coordinate of $S^1$.
  Hence 
  it suffices to show that 
    \begin{align*}
   &  \sum_{\substack{P \in \Lambda_{\max} \\ Q \subset P}}   
     m_P \int_{ \exp ( - P)_\C  \cap ( - \log \lvert \cdot \rvert )^{-1} ( Q) } 
     \Re (D \big( \frac{1}{f_{P,1} ( - \log \lvert z \rvert ) } \omega_P (E_I' , E_J'' )
        \overline{\omega_P( \N''_P, E_I',E_{J \setminus \{j \} }'') }
       \big) (E_j') )  \\
   &  d \theta_{P,1} \wedge \mu_{ \exp ( - Q )_\C} \\
   = & 0. 
  \end{align*}  
  By the condition on $E_i'$ (\cref{n max for strange superframe}), 
  the integrands are functions of $-\log \lvert z \rvert $.
  In particular, by Fubini's theorem,
  it suffices to show that 
  \begin{align*}
   & \sum_{ \substack{P \in \Lambda_{\max} \\ Q \subset P}}   
      \frac{m_P}{f_{P,1} ( - \log \lvert z \rvert )} \omega_P (E_I' , E_J'' )
      \overline{\omega_P( \N''_P, E_I',E_{J \setminus \{j \} }'') } 
      = 0  
  \end{align*}
  on $\exp (-Q)_\C$.
  This follows from 
    condition $\t''_{\min}$ for $\omega$ 
    (which means ``condition $\t''_{\min}$ at $Q$'' 
    for $(\omega_{P} (E_j''))_{\substack{P \in \Lambda_{\max} \\ Q \subset P}}$), condition $\n''_{\max}$ for $\omega$,
   and \cref{n max for strange superframe}.
\end{proof}

Since smooth endomorphisms $\mathcal{R}^W$ and $S_{P,Q}$ 
are bounded, 
we have the following.

\begin{cor}[Gaffney-G{\aa}rding's inequality]\label{Gaffney-Garding's inequality}
  There exists $C >0$ 
  such that for any $\omega \in W_{\Trop, \t''_{\min}, \n''_{\max}}^{p,q,1,2} ( \Lambda)$,
  we have 
 $$\lVert \omega\rVert_{W^{1,2} ( \Lambda )}^2 
   \leq   C ( \lVert \omega \rVert_{L^2 ( \Lambda )}^2 + \D(\omega,\omega) ).$$
\end{cor}
\begin{proof}
  The assertion follows from \cref{expression of Sobolev 1 norm}
  and 
  Ehrling's inequality (\cite[Lemma 1.5.3]{Sch95}): 
  for each $\epsilon >0$,
  there exists $C_\epsilon >0 $ 
  such that 
    $$\lVert \omega_P|_{\partial \exp ( - P)_\C  } 
    \rVert_{L^2 (\partial \exp ( - P)_\C  ) )} \leq \epsilon  \lVert \omega_P \rVert_{W^{1,2}( \exp ( - P)_\C  ) } + C_\epsilon \lVert \omega_P \rVert_{L^2(\exp ( - P)_\C )}$$
  ($P \in \Lambda_{\max}$).
  (Here to apply Ehrling's inequality, we use the fact that 
   the trace operator $$W^{1,2}( \exp ( - P)_\C  ) \to L^{2} (\partial \exp ( - P)_\C  ) $$
   is compact.)
\end{proof}

By applying $\overline{*}$-operator, we also have the following.
\begin{cor}\label{* of Gaffney-Garding's inequality}
  There exists $C >0$ 
  such that for any $\omega \in W_{\Trop, \n''_{\min}, \t''_{\max}}^{p,q,1,2} ( \Lambda)$,
  we have 
 $$\lVert \omega\rVert_{W^{1,2} ( \Lambda )}^2 
   \leq   C ( \lVert \omega \rVert_{L^2 ( \Lambda )}^2 + \D(\omega,\omega) ).$$
\end{cor}

\section{Dirichlet and Neumann potentials}
In this section, we study Dirichlet and Neumann potentials.
This section is almost parallel to \cite[Section 2.2]{Sch95}.
We often omit proofs.
We put 
\begin{align*}
  \H^{p,q} ( \Lambda ) & := \{ \omega \in W_{\Trop, \t''_{\max}, \n''_{\max} }^{p,q,1,2} ( \Lambda ) \mid \overline{\partial} \omega_P = 0, \ \overline{\delta} \omega_P = 0  \ ( P \in \Lambda_{\max} ) \} , \\
  \H_D^{p,q} ( \Lambda ) & := \H^{p,q} ( \Lambda ) \cap W_{\Trop,\t''_{\min}, \n''_{\max} }^{p,q,1,2} ( \Lambda ) , \\
  \H_N^{p,q} ( \Lambda ) & := \H^{p,q} ( \Lambda ) \cap W_{\Trop,\t''_{\max}, \n''_{\min} }^{p,q,1,2} ( \Lambda ) .
\end{align*}
The elements of $\H_D^{p,q} ( \Lambda )$ are called \emph{Dirichlet fields} and 
those of $\H_N^{p,q} ( \Lambda )$ are called \emph{Neumann fields}.

\begin{rem}[Rellich's lemma]\label{Rellich's lemma}
  For any $P \in \Lambda_{\max}$,
  the natural map
   $$W^{p,q,1,2} ( \relint ( \exp ( - P)_\C  ) ) \to L^{p,q,2} ( \relint ( \exp ( - P)_\C  ) )$$ 
  is compact (\cite[Theorem 7.2]{Wlo87}).
\end{rem}

\begin{thm}\label{finite dimensionality of Dirichlet fields}
  $$\dim \H_D^{p,q} ( \Lambda ) < \infty .$$
\end{thm}
\begin{proof}
  In the same way as \cite[Theorem 2.2.2]{Sch95}, 
  this follows from Gaffney-G{\aa}rding's inequality (\cref{Gaffney-Garding's inequality}) and Rellich's lemma (\cref{Rellich's lemma}).
\end{proof}

In particular, the space $\H_D^{p,q} ( \Lambda )$ is closed. 
We have $L^2$-orthogonal decomposition 
$$ L^{p,q,2} ( \Lambda ) = \H_D^{p,q} ( \Lambda ) \oplus \H_D^{p,q} ( \Lambda )^{\perp}.$$
We put 
$$ \H_D^{p,q} ( \Lambda )^{\circledast} := W_{\Trop,\t''_{\min}, \n''_{\max} }^{p,q,1,2} ( \Lambda ) \cap \H_D^{p,q} ( \Lambda )^{\perp},$$
which is closed in $W_{\Trop,\t''_{\min}, \n''_{\max} }^{p,q,1,2} ( \Lambda ) $, hence a Hilbert subspace of $W_{\Trop,\t''_{\min}, \n''_{\max} }^{p,q,1,2} ( \Lambda )$.

\begin{prp}\label{D is W12 elliptic on the Hilbert subspace}
  There exists $c_1,c_2 >0$ 
   such that 
    $$c_1 \Arrowvert \omega \Arrowvert_{W^{1,2} ( \Lambda )}^2 
      \leq D ( \omega, \omega )
      \leq c_2 \Arrowvert \omega \Arrowvert_{W^{1,2} ( \Lambda )}^2 
      \quad ( \omega \in \H_D^{p,q} ( \Lambda )^{\circledast} ).$$
\end{prp}
\begin{proof}
  In the same way as \cite[Proposition 2.2.3]{Sch95},
  this follows from Gaffney-G{\aa}rding's inequality (\cref{Gaffney-Garding's inequality}), Rellich's lemma (\cref{Rellich's lemma}), and functional analysis.
\end{proof}

\begin{thm}\label{existence and uniqueness of Dirichlet potential }
 For each $\eta \in \H_D^{p,q} ( \Lambda )^{\perp}$,
  there exists 
   a unique element $\phi_D \in \H_D^{p,q} ( \Lambda )^{\circledast}$
  such that 
   $$ \langle \overline{\partial} \phi_D, \overline{\partial} \xi \rangle_{L^{2} ( \Lambda ) }  
    + \langle \overline{\delta} \phi_D, \overline{\delta} \xi \rangle_{L^{2} ( \Lambda ) }  
    = \langle \eta , \xi \rangle_{ L^{2} ( \Lambda ) }
    \quad ( \xi \in W_{\Trop,\t''_{\min}, \n''_{\max}}^{p,q,1,2} ( \Lambda ) ).$$
\end{thm}
We call $\phi_D$ the \emph{Dirichlet potential} of $\eta$.
\begin{proof}
  The proof is same as \cite[Theorem 2.2.4]{Sch95}.
  The existence of a Dirichlet potential $\phi_D$ follows from 
    \cref{D is W12 elliptic on the Hilbert subspace}
   and 
    the Lax-Milgram lemma (\cite[Corollary 1.5.10]{Sch95}),
  and 
  the uniquness is easy.
\end{proof}

In the rest of this paper, 
to proceed with our tropcial Hodge theory,
we use the following assumption on regularity of Dirichlet potentials.

\begin{asm}\label{assumption regularity of Derichlet potentials}
\begin{itemize}
  \item $ \H_D^{p,q} ( \Lambda )  \subset \A_{\Trop,\pqs}^{p,q} ( \Lambda )$.
  \item For $\eta \in \H_D^{p,q} ( \Lambda )^{\perp} \cap \A_{\Trop,\pqs}^{p,q} ( \Lambda )$,
  we have 
   $\phi_D \in  \H_D^{p,q} ( \Lambda )^{\circledast} \cap \A_{\Trop,\pqs}^{p,q} ( \Lambda ) . $
 \end{itemize}
\end{asm}

\begin{rem}\label{remark on regularity of solutions of Lalpacians elliptic partial differential operators}
\cref{assumption regularity of Derichlet potentials} 
is an analog of regularity of weak solutions of Laplacian 
with Dirichlet boundary condition
on compact Riemann manifolds with smooth boundaries 
(\cite[Theorem 2.2.6]{Sch95}).
Proof in \cite{Sch95} is devided into 2 parts:
least (i.e., 2 times) regularity and full regularity.

To prove least regularity, 
an elementary method,
\emph{differential quotient}, 
is used (\cite[Corollary 2.3.4]{Sch95}), see also  
\cite[Chapter 5, Proposition 7.2]{Tay96}, 
\cite[Theorem 15.1 and Theorem 15.3]{Wlo87} (function case).
This method is not directly applicapable in our tropical case.

For full regularity,
the theory of \emph{symbols} and \emph{psuedo differential operators} is used,
see \cite[Theorem 20.1.2 and proof of Theorem 20.1.8]{Hor94}
and also references in \cite[below Theorem 1.6.2]{Sch95}.
\end{rem}

We do not have (sufficiently) non-trivial examples for which \cref{assumption regularity of Derichlet potentials} holds.

\begin{rem}\label{Counter example of Assumption}
  The natural compactification $\overline{\Lambda_{BH}}$ of a $2$-dimensional tropical fan $\Lambda_{BH}$ constructed by Babaee-Huh \cite[Section 5]{BH17}
  is a tropically compact projective tropical vareity,
  and $\Q$-smooth in codimension $1$, but
  it does NOT SATISFY \cref{assumption regularity of Derichlet potentials}.
  Namely, 
  by \cref{the Kahler package for tropical Dolbeault cohomology},
  if it satisfies \cref{assumption regularity of Derichlet potentials},
  then $H_{\Trop,\Dol,\pqs,\t''_{min}}^{1,1}(\overline{\Lambda_{BH}})$
  satisfies the Hodge index theorem.
  However, Babaee-Huh proved that 
  $$\CH^1 ( T_{\Lambda_{BH}}) \otimes \R \cong H_{\Trop,\Dol}^{1,1}(\overline{\Lambda_{BH}}) (\subset H_{\Trop,\Dol,\pqs,\t''_{min}}^{1,1}(\overline{\Lambda_{BH}}) )$$ 
  has $3$-dimensional positive definite subspace, which is a contraction (where $T_{\Lambda_{BH}}$ is the toric variety corresponding to the fan $\Lambda_{BH}$).
\end{rem}

Under \cref{assumption regularity of Derichlet potentials},
for $\eta \in \H_D^{p,q} ( \Lambda )^{\perp} \cap \A_{\Trop,\pqs}^{p,q} ( \Lambda )$, 
  a piece-wise quasi-smooth superform $\Delta'' \phi_D \in \A_{\Trop,\pqs}^{p,q} ( \Lambda )$ is well-defined,
and 
  by Green's formula,
  we have 
   \begin{align}\label{strong version of Dirichlet potential} 
    & \langle \Delta'' \phi_D , \xi \rangle_{ L^{2} ( \Lambda ) }
     + \sum_{P \in \Lambda_{\max}} m_P \int_{\partial  \exp ( - P)_\C   }
       \big( \overline{ \xi \wedge \overline{*} \overline{ \partial } \phi_D } 
        - \overline{ \delta} \phi_D  \wedge \overline{*} \xi   \big) 
    =  \langle \eta , \xi \rangle_{ L^{2} ( \Lambda ) } 
   \end{align}
  $ ( \xi \in W_{\Trop,\t''_{\min}, \n''_{\max}}^{p,q,1,2} ( \Lambda ) ) .$

\begin{thm}\label{Dirichlet potential solution of boundary value problem}
  Under \cref{assumption regularity of Derichlet potentials},
  \begin{itemize}
    \item the differential form 
   $\overline{\delta } \phi_D$ satisfies condition $\t''_{\min}$,
    \item the differential form 
   $\overline{\partial } \phi_D$ satisfies condition $\n''_{\max}$,
      and 
    \item we have $\Delta'' \phi_D = \eta   \in \A_{\Trop,\pqs}^{p,q} ( \Lambda )$.
  \end{itemize}
\end{thm}
\begin{proof}
  This follows from \eqref{strong version of Dirichlet potential} and density of $W_{\Trop,\t''_{\min}, \n''_{\max}}^{p,q,1,2} ( \Lambda ) \subset L_{\Trop}^{p,q,2} ( \Lambda )$.
\end{proof}

We put
$$ \H_N^{p,q} ( \Lambda )^{\circledast} := W_{\Trop,\t''_{\max}, \n''_{\min} }^{p,q,1,2} ( \Lambda ) \cap \H_N^{p,q} ( \Lambda )^{\perp}.$$

By applying $\overline{*}$-operator, we also have the following.
\begin{thm}\label{On Neumann potential and fields}
  \begin{itemize}
   \item $ \H_N^{p,q} ( \Lambda ) $ is finite dimensional.
   \item For each $\eta \in \H_N^{p,q} ( \Lambda )^{\perp}$,
  there exists 
         a unique element $\phi_N \in \H_N^{p,q} ( \Lambda )^{\circledast}$
        such that 
         $$ \langle \overline{\partial} \phi_N, \overline{\partial} \xi \rangle_{L^{2} ( \Lambda ) }  
          + \langle \overline{\delta} \phi_N, \overline{\delta} \xi \rangle_{L^{2} ( \Lambda ) }  
          = \langle \eta , \xi \rangle_{ L^{2} ( \Lambda ) }
          \quad ( \xi \in W_{\Trop,\t''_{\max}, \n''_{\min} }^{p,q,1,2} ( \Lambda ) ).$$
   \item 
  Under \cref{assumption regularity of Derichlet potentials},
  for $\eta \in \H_N^{p,q} ( \Lambda )^{\perp}\cap \A_{\Trop,\pqs}^{p,q}(\Lambda)$,
  we have 
   $$ \langle \Delta'' \phi_N , \xi \rangle_{ L^{2} ( \Lambda ) }
     + \sum_{P \in \Lambda_{\max}} m_P \int_{\partial  \exp ( - P)_\C   }
       \big( \overline{  \xi \wedge \overline{*}  \overline{ \partial } \phi_N } 
        -  \overline{ \delta} \phi_N  \wedge \overline{*} \xi   \big) 
     = \langle \eta , \xi \rangle_{ L^{2} ( \Lambda ) } $$
  $ ( \xi \in W_{\Trop,\t''_{\max}, \n''_{\min} }^{p,q,1,2} ( \Lambda ) ) .$
  
  In particular, 
  \begin{itemize}
    \item the differential form 
   $\overline{\delta } \phi_N$ satisfies condition $\t''_{\max}$,
    \item the differential form 
   $\overline{\partial } \phi_N$ satisfies condition $\n''_{\min}$,
      and 
    \item 
    we have $\Delta'' \phi_N = \eta   \in \A_{\Trop,\pqs}^{p,q} ( \Lambda )$.
  \end{itemize}
  \end{itemize}
\end{thm}
We call $\phi_N$ the \emph{Neumann potential} of $\eta$.

\section{The Hodge-Morrey-Friedrichs decomposition and the Hodge isomorphism}
  In this section, we suppose that  \cref{assumption regularity of Derichlet potentials} hold,
  and we shall prove the Hodge-Morrey-Friedrichs decomposition (\cref{the Hodge-Morrey-Friedrichs decomposition})
  and the Hodge isomorphism (\cref{the Hodge isomorphism})
  in a parallel way to \cite[Section 2.4 and 2.6]{Sch95}.
\begin{dfn}
   We put
  \begin{align*}
    \E^{p,q} ( \Lambda ) 
    & := \{ \overline{ \partial } \alpha \mid \alpha \in \A_{\Trop,\pqs , \t''_{\min} }^{p,q-1} ( \Lambda ) \} \subset \A_{\Trop,\pqs}^{p,q} ( \Lambda )   \\
    \CC^{p,q} ( \Lambda )
    & := \{ \overline{ \delta } \beta \mid \beta \in \A_{\Trop,\pqs , \n''_{\min} }^{p,q+1} ( \Lambda ) \} \subset \A_{\Trop,\pqs}^{p,q} ( \Lambda )  .
  \end{align*}
\end{dfn}

\begin{thm}[the Hodge-Morrey decompostion]\label{Hodge-Morrey decomposition}
  Under \cref{assumption regularity of Derichlet potentials},
  we have an $L^2$-orthogonal decomposition
  \begin{align*}
  \A_{\Trop,\pqs}^{p,q} ( \Lambda )
   & =  \E^{p,q} ( \Lambda ) 
     \oplus  \CC^{p,q} ( \Lambda )
     \oplus (\H^{p,q} ( \Lambda )   \cap \A_{\Trop,\pqs}^{p,q} ( \Lambda ) ) .
  \end{align*}
\end{thm}

\begin{lem}\label{cor 2.1.3}
  Let $\omega \in \A_{\Trop,\pqs}^{p,q - 1 } ( \Lambda )$ 
      and 
      $\eta \in \A_{\Trop, \pqs}^{p,q + 1} ( \Lambda )$.
  We assume that 
  \begin{itemize}
    \item $\omega $ satisfies condition $\t''_{\min}$ 
          and $\eta$ satisfies condition $\n''_{\max}$, 
      or 
    \item $\omega $ satisfies condition $\t''_{\max}$ 
          and $\eta$ satisfies condition $\n''_{\min}$.  
  \end{itemize}
  Then we have 
    $$ \langle \overline{\partial} \omega , \overline{\delta} \eta \rangle_{L^2 ( \Lambda ) } = 0 .$$
\end{lem}
\begin{proof}
  We prove the first case.  The other case is similar.
  Since $\overline{\partial}\omega$ satisfies condition $\t''_{\min}$ 
    (\cref{condition t min max preserved by partial}),
  by Green's formula  
    for $\overline{\partial} \omega$ and $ \eta$,
  we have 
    $ \langle \overline{\partial} \omega , \overline{\delta} \eta \rangle_{L^2 ( \Lambda ) }  = 0.$
\end{proof}

\begin{cor}\label{L2 orthogonality of Hodge Morrey decomposition}
  The direct summands in \cref{Hodge-Morrey decomposition}
  are $L^2$-orthogonal each other.
\end{cor}
\begin{proof}
  This follows from \cref{cor 2.1.3} 
            and Green's formula (\cref{Green's formula}).
\end{proof}

\begin{proof}[Proof of \cref{Hodge-Morrey decomposition}]
  We have decompositions 
    $$ \omega = \lambda_D + ( \omega - \lambda_D)
              = \lambda_N + ( \omega - \lambda_N)$$
    ($\lambda_D \in \H_D^{p,q} ( \Lambda )$,
     $\omega - \lambda_D \in \H_D^{p,q} ( \Lambda )^{\perp}$,
     $\lambda_N \in \H_N^{p,q} ( \Lambda )$,
     $\omega - \lambda_N \in \H_N^{p,q} ( \Lambda )^{\perp}$.)
  Let $\phi_D$ (resp. $\phi_N$) be 
    the Dirichlet (resp. Neumann) potential of $\omega - \lambda_D$ (resp. $\omega - \lambda_N$).
  By  
  \cref{assumption regularity of Derichlet potentials}, 
  \cref{Dirichlet potential solution of boundary value problem}, 
  and
     \cref{On Neumann potential and fields},
  we have 
    $\alpha_\omega := \overline{\delta} \phi_D \in \A_{\Trop,\pqs, \t''_{\min} }^{p,q-1} ( \Lambda ) $
    and
    $\beta_\omega := \overline{\partial} \phi_N \in \A_{\Trop, \pqs, \n''_{\min} }^{p,q+1} ( \Lambda ) $.
  We put 
    $\kappa_\omega := \omega - \overline{\partial} \alpha_\omega - \overline{\delta} \beta_\omega.$

    It suffices to show that $\kappa_\omega \in \H^{p,q} ( \Lambda ) \cap \A_{\Trop,\pqs}^{p,q} ( \Lambda )$.
  By \cref{Dirichlet potential solution of boundary value problem}, 
     \cref{On Neumann potential and fields},
     and \cref{cor 2.1.3},
  we have $ \kappa_\omega \in ( \E^{p,q} ( \Lambda ) \oplus \CC^{p,q} ( \Lambda ) )^{\perp}$.
  We have 
  \begin{align*}
    0 & = \langle \kappa_\omega , \overline{\partial} f \rangle_{L^2 ( \Lambda ) } 
        = \langle \overline{\delta} \kappa_\omega , f \rangle_{L^2 ( \Lambda ) }  
        \quad  \bigg( f \in \bigoplus_{ P \in \Lambda_{\max} } \A_{\Trop,\qs, 0 }^{p,q-1} ( \relint ( \exp ( - P)_\C  ) ) \bigg) ,  \\
    0 & = \langle \kappa_\omega , \overline{\delta} g \rangle_{L^2 ( \Lambda ) } 
        = \langle \overline{\partial} \kappa_\omega ,  g \rangle_{L^2 ( \Lambda ) }  
        \quad  \bigg( g \in \bigoplus_{ P \in \Lambda_{\max} } \A_{\Trop, \qs, 0 }^{p,q+1} ( \relint ( \exp ( - P)_\C  ) ) \bigg)  ,
  \end{align*}
  where $\A_{\Trop, \qs, 0 }^{p,*} ( \relint ( \exp ( - P)_\C  ) ) $ is the set of compactly supported quasi-smooth superforms on $\relint ( \exp ( - P)_\C  )$. 
  Hence we have 
    $ \overline{\delta} \kappa_\omega =0$ 
   and 
    $ \overline{\partial} \kappa_\omega =0$.
  Moreover, for $\alpha \in \A_{\Trop,\qs}^{p,q-1} ( \Lambda )$,
  we have
  \begin{align*}
    0  = \langle \kappa_\omega , \overline{\partial} \alpha \rangle_{L^2 ( \Lambda ) } 
       & = \langle \overline{\delta} \kappa_\omega , \alpha \rangle_{L^2 ( \Lambda ) }  
         + \sum_{ P \in \Lambda_{\max} } m_P 
           \overline{\int_{\partial \exp ( - P)_\C  }
            \alpha_P \wedge \overline{*}  \kappa_{\omega,P} 
           }  \\
      & =  \sum_{ P \in \Lambda_{\max} } m_P 
           \overline{\int_{\partial \exp ( - P)_\C  }
            \alpha_P \wedge \overline{*}  \kappa_{\omega,P} 
           }    .
  \end{align*}
  Hence $\kappa_\omega$ satisfies condition $\n''_{\max}$.
  Similarly, it also satisfies condition $\t''_{\max}$.
  Consequently, we have $ \kappa_\omega \in \H^{p,q} ( \Lambda ) \cap \A_{\Trop,\pqs}^{p,q} ( \Lambda )$.
\end{proof}

\begin{thm}[the Friedrichs decomposition]\label{Friedrichs decomposition}
  Under \cref{assumption regularity of Derichlet potentials},
 we have an $L^2$-orthogonal decomposition
 \begin{align*}
  \H^{p,q} ( \Lambda ) \cap \A_{\Trop, \pqs}^{p,q} ( \Lambda ) & = \H_D^{p,q} ( \Lambda ) 
  \oplus  \{ \overline{ \delta} \gamma \in \H^{p,q} ( \Lambda )\mid \gamma \in \A_{\Trop, \pqs, \n''_{\max} }^{p,q+1} ( \Lambda ) \}    . 
 \end{align*}
\end{thm}
\begin{proof}
  $L^2$-orthogonality follows from Green's formula (\cref{Green's formula}).
  Let 
    $$\kappa \in \H^{p,q} ( \Lambda ) \cap \A_{\Trop, \pqs}^{p,q} ( \Lambda ) \cap \H_D^{p,q} ( \Lambda )^{\perp} ,$$
  and $\phi_\kappa $ the Dirichlet potential of $\kappa$.
  We put $\gamma_\kappa := \overline{\partial} \phi_\kappa \in \A_{\Trop, \pqs, \n''_{\max} }^{p,q+1} ( \Lambda )$.
  We have 
    $\kappa - \overline{\delta}\gamma_\kappa \in \H_D^{p,q} ( \Lambda )^{\perp}.$
  We also have 
    $\kappa - \overline{\delta}\gamma_\kappa \in \H_D^{p,q} ( \Lambda ),$
  i.e.,
  \begin{itemize}
    \item $\overline{\partial} ( \kappa - \overline{\delta}\gamma_\kappa ) =0, 
          \  \overline{\delta} ( \kappa - \overline{\delta}\gamma_\kappa )
          =0$,
    \item 
      by \cref{condition t min max preserved by partial} and \cref{Dirichlet potential solution of boundary value problem},
      the differential form
      $\kappa - \overline{\delta}\gamma_\kappa = \overline{\partial} \overline{\delta} \phi_\kappa$ satisfies condition $\t''_{\min}$, 
  and 
  \item similarly, the differential form
      $\kappa - \overline{\delta}\gamma_\kappa = \kappa - \overline{\delta}\overline{\partial} \phi_\kappa $ satisfies condition $\n''_{\max}$.
  \end{itemize}
  Hence 
    $\kappa = \overline{\delta}\gamma_\kappa .$
\end{proof}

\begin{cor}[the Hodge-Morrey-Friedrichs decomposition]\label{the Hodge-Morrey-Friedrichs decomposition}
  We have 
  \begin{align*}
  \A_{\Trop,\pqs}^{p,q} ( \Lambda )
   & = \E^{p,q} ( \Lambda ) 
     \oplus  \CC^{p,q} ( \Lambda )
    \oplus \H_D^{p,q} ( \Lambda ) 
  \oplus  \{ \overline{ \delta} \gamma \in \H^{p,q} ( \Lambda )\mid \gamma \in \A_{\Trop, \pqs, \n''_{\max} }^{p,q+1} ( \Lambda ) \}    .
  \end{align*}
\end{cor}

\begin{thm}[the Hodge isomorphism]\label{the Hodge isomorphism}
  We have 
   $$ H^q( ( \A_{\Trop,\pqs,\t''_{\min} }^{p,*} ( \Lambda ), \overline{\partial} )) \cong \H_D^{p,q} ( \Lambda ) .$$
\end{thm}
\begin{proof}
  Let $ \omega \in \A_{\Trop, \pqs,\t''_{\min} }^{p,q} ( \Lambda )$.
  Then $\overline{\partial } \omega =0$ 
  if and only if 
   $$\omega \in \E^{p,q} ( \Lambda ) 
     \oplus \H_D^{p,q} ( \Lambda ) 
   \oplus  \{ \overline{ \delta} \gamma \in \H^{p,q} ( \Lambda )\mid \gamma \in \A_{\Trop, \pqs, \n''_{\max} }^{p,q+1} ( \Lambda ) \}   . $$
  In this case, the 
   $$  \{ \overline{ \delta} \gamma \in \H^{p,q} ( \Lambda )\mid \gamma \in \A_{\Trop, \pqs, \n''_{\max} }^{p,q+1} ( \Lambda ) \}   \text{-component } \overline{\delta} \gamma_{\omega} $$
  of $\omega$ satisfies condition $\t''_{\min}$, i.e., $\overline{\delta} \gamma_{\omega} \in \H_D^{p,q} ( \Lambda )$, hence $\overline{\delta} \gamma_{\omega} =0$.
  The assertion holds.
\end{proof}

\section{The K\"{a}hler package}
In this section, we assume that \cref{assumption regularity of Derichlet potentials} holds,
and we shall prove 
  the K\"{a}hler package in the same way as compact K\"{a}hler manifolds, see \cite{GH78} and \cite{Wei58}.

We put
  $$\H_{\min }^{p,q} ( \Lambda ) := \H^{p,q} ( \Lambda ) \cap W_{\Trop,\t''_{\min}, \n''_{\min}}^{p,q,1,2}( \Lambda ).$$
For $ \omega \in L_{\Trop}^{p,q,2} ( \Lambda ) $,
we put 
\begin{align*}
  L \omega & := \omega_{Kah} \wedge \omega ,\\
  \Lambda_L \omega & :=  (-1)^{ p + q } \overline{*}  L   \overline{*} \omega .
\end{align*}

\begin{rem}\label{harmonic forms and operator}
  \begin{itemize}
    \item 
Since 
$$\overline{*} \colon W_{\Trop}^{p,q,1,2} (\Lambda ) \cong W_{\Trop}^{d-p,d-q,1,2} (\Lambda )$$ 
is an isometry,
we have 
  $$ \overline{*} \colon \H_{\min}^{p,q} ( \Lambda )
     \cong \H_{\min}^{d-p,d-q} ( \Lambda ).$$
  \item
Let $\omega$ be a piece-wise quasi-smooth superform
such that 
$\omega$ satisfies
condition $\t''_{\min}$ and $\n''_{\min}$,
$\overline{\partial} \omega$
satisfies condition $\n''_{\max}$, and
$\overline{\delta} \omega$
satisfies condition $\t''_{\max}$.
Then it
is $\overline{\partial}$-closed and $\overline{\delta}$-closed if and only if it is $\Delta''$-closed.   
We have 
$ \Delta'' L = L \Delta''$
(\cite[page 115]{GH78}).
  \end{itemize}
\end{rem}

\begin{lem}\label{support for L preserving harmonic forms}
  Let $\omega \in \H_{\min}^{p,q} ( \Lambda )$.
  Then 
    $ \overline{\delta} L \omega$ satisfies condition $\t''_{\min}$.
\end{lem}
\begin{proof}
 We have $[L,\overline{\delta}] = - i \partial$, where $\partial $ is the holomorphic differential ($d = \partial + \overline{\partial}$).
 We fix $Q \in \Lambda_{\max -1}^{\ess}$ and 
 a basis $x_1,\dots,x_n \in M$ such that 
 the affine span of $Q \cap \R^n$ is 
 $$\{ (x_1,\dots,x_n) \in \R^n \mid x_i =a_i \ ( d \leq i \leq n) \}  $$
 ($ a_i \in \R$).
 We put $z_1 :=x_1,\dots,z_n:=x_n \in M$ as a coordinate of $(\C^{\times})^n$.
 For each $P \in \Lambda_{\max}$ containing $Q$,
 we fix $z_P= \prod_{i =d}^n z_i^{b_{P,i}}$ $(b_{P,i} \in \Z)$ such that $\frac{\partial}{\partial \log z_P} \in T \exp (-P)_\C \setminus \{0\}$.
 Then 
 it suffices to show that 
 for $c_P \in \Z$, $I \in \{1,\dots,d-1\}^p, $ and $ J \in \{1,\dots,d-1\}^q$
 such that 
 $ \sum_{\substack{P \in \Lambda_{\max} \\ Q \subset P}}  c_P 
  \frac{\partial}{\partial \log z_P} = 0$, 
  we have
 $$ \sum_{\substack{P \in \Lambda_{\max} \\ Q \subset P}} 
   c_P
   \partial \omega_P \bigg(\frac{\partial}{\partial \log z_P}, \frac{\partial}{\partial \log z_I}, \frac{\partial}{\partial \log \overline{z_J}}\bigg) 
   = 0 $$
 on $\exp (-Q)_\C \cap (\C^{\times})^n$. 
 This follows from
 conditions $\t''_{\min}$ and $\n''_{\min}$ for $\omega$
 and equalities
 \begin{align*}
  & D\bigg( \omega_P \bigg( \frac{\partial}{\partial \log z_I}, \frac{\partial}{\partial \log \overline{z_J}}\bigg)\bigg) \bigg(\frac{\partial}{\partial \log z_1} \bigg) \\
  = &  D\bigg( \omega_P \bigg( \frac{\partial}{\partial \log z_I}, \frac{\partial}{\partial \log \overline{z_J}}\bigg)\bigg) \bigg(\frac{\partial}{\partial \log \overline{z_1}} \bigg) \\
  =& \sum_{l =1}^q (-1)^q D\bigg( \omega_P \bigg( \frac{\partial}{\partial \log z_I}, \frac{\partial}{\partial \log \overline{z_1}}, \frac{\partial}{\partial \log \overline{z_{J \setminus \{j_l \} }}} \bigg)\bigg) \bigg(\frac{\partial}{\partial \log \overline{z_{j_l}}} \bigg),
 \end{align*}
 where 
 the first equality follows from the fact that 
  $\omega_P \big( \frac{\partial}{\partial \log z_I}, \frac{\partial}{\partial \log z_J}\big)$
  is a function on $- \log \lvert z \rvert$,
  and 
 the last equality follows from $\overline{\partial} \omega=0$.
(We use the following trivial fact to show the assertion. We put 
$$\frac{\partial}{\partial \log \overline{z_P}}
= \frac{1}{f_P(- \log \lvert z \rvert ) } \N''
  + \sum_{i =1}^{d-1} g_{P,i}( - \log \lvert z \rvert ) 
  \frac{\partial}{\partial \log \overline{z_i}}  .$$
  Then  we have 
 $$ \sum_{\substack{P \in \Lambda_{\max} \\ Q \subset P}}   
  \frac{c_P}{f_P(- \log \lvert z \rvert ) } \N' = 0,$$ 
and 
 $$ \sum_{\substack{P \in \Lambda_{\max} \\ Q \subset P}}  c_P 
  \sum_{i =1}^{d-1} g_{P,i}( - \log \lvert z \rvert ) 
  \frac{\partial}{\partial \log \overline{z_i}} = 0.)$$
\end{proof}

\begin{lem}\label{ * L Lambda preserve harmonic forms}
  Let $\omega \in \H_{\min}^{p,q} ( \Lambda )$.
  Then 
    $  L \omega$ and $ \Lambda_L \omega $ are in $ \H_{\min}^{p,q} ( \Lambda ).$
\end{lem}

\begin{proof}
  By  \cref{harmonic forms and operator} and \cref{support for L preserving harmonic forms},
  it suffices to show that 
    $L \omega $ satisfies condition $\n''_{\min}$.
  For a local $g$-orthogonal normal quasi-smooth $(1,0)$-superframe $E_i':= E_{P,Q,i}'$ as in \cref{n max explicitly},
  we have 
  \begin{align*}
    \overline{*} ( \omega_{Kah} \wedge \omega_P ) 
      ( \N', E_{I}',E_J'' )
      = & i^{d^2}  \epsilon_{I,J} \omega_{Kah} \wedge \omega_P  ( E_{ \{1,\dots,d-1\} \setminus I}', \N'', E_{\{1,\dots,d-1\} \setminus  J}'' )  \\
      = & i \sum_{j \notin I \cup J} \epsilon_{I,J,j}  \overline{*} \omega_P  (\N', E_{ I \cup \{j\} }',  E_{J \cup \{j\} }'' )  \\
    \overline{*} ( \omega_{Kah} \wedge \omega_P ) 
      ( E_K',E_J'' )
      = & i^{d^2} \epsilon_{K,J}' \omega_{Kah} \wedge \omega_P  (\N', E_{ \{1,\dots,d-1\} \setminus K}', \N'', E_{\{1,\dots,d-1\} \setminus  J}'' )  \\
      = & i^{d^2 -1 }\epsilon_{K,J}'' \omega_P  ( E_{ \{1,\dots,d-1\} \setminus K}',E_{\{1,\dots,d-1\} \setminus  J}'' )\\
      & + i  \sum_{j \notin K \cup J} \epsilon_{K,J,j}' \overline{*} \omega_P  (E_{ K \cup \{j\} }',  E_{J \cup \{j\} }'' )  
  \end{align*}
  for $I \in \{1,\dots,d-1\}^{d-p-2},K \in \{1,\dots,d-1\}^{d-p-1},$ and $J \in \{1,\dots,d-1\}^{d-q-1}$, 
  where $\epsilon_{I,J}, \epsilon_{K,J}', \epsilon_{K,J}'', \epsilon_{I,J,i}, $ $\epsilon_{K,J,i}' \in \{1,-1\}$.
  Hence $L \omega $ satisfies condition $\n''_{\min}$.
\end{proof}

For $p $ and $q$ with $p + q \leq d$,
we put 
  $$\Prim\H_{\min }^{p,q} ( \Lambda ) 
    := \ker ( \Lambda_L \colon \H_{\min}^{p,q} ( \Lambda ) \to \H_{\min}^{ p - 1 , q - 1 } ( \Lambda )  ).$$

\begin{cor}\label{hard Lefschetz}
  For $p $ and $q$ with $p + q \leq d$,
  we have an $L^2$-orthogonal decomposition
  $$ \H_{\min}^{p,q} ( \Lambda ) 
    = \bigoplus_{ i = 0}^{ \min \{ p, q \} }  
      L^i \Prim \H_{\min }^{p-i,q-i} ( \Lambda )   $$
  and an isomorphism
  $$ L^{ d - p - q } \colon \H_{\min}^{p,q} ( \Lambda ) 
    \cong  \H_{\min}^{d - q, d - p} ( \Lambda ) .$$
\end{cor}
\begin{proof}
  This follows from \cref{ * L Lambda preserve harmonic forms} and 
  the similar decomposition and isomorphism for differential forms at each point on K\"{a}hler manifolds.
\end{proof}

\begin{thm}[The Hodge-Riemann bilinear relation]\label{Hodge-Riemann bilinear relation}
  For $p$ and $q$ with $p +q \leq d$,
  a symmetric $\R$-bilinear map
  $$\Prim\H_{\min }^{p,q} ( \Lambda ) \times \Prim\H_{\min }^{p,q} ( \Lambda ) \ni (\omega,\eta) 
    \mapsto i^{p-q} (-1)^{\frac{(p +q)(p+q -1) }{2} }  
    \int_{\Lambda} L^{ d - p - q } \omega \wedge \overline{\eta} \in \R$$
  is positive definite.
\end{thm}
\begin{proof}
  For $\omega \in \Prim\H_{\min }^{p,q} ( \Lambda )$, 
  by computation of linear algebra (\cite[Chapter 1, Theorem 2]{Wei58}),
  we have 
  $$* \omega = i^{p-q} (-1)^{\frac{(p +q)(p+q +1) }{2} } \frac{1}{ ( d - p -q )! } L^{ d - p - q } \omega .$$
  Hence 
  \begin{align*}
     0< \Arrowvert \omega \Arrowvert_{L^2 ( \Lambda ) }^2  
   = & \int_{\Lambda} \omega \wedge \overline{*} \omega  \\
   = & \int_{\Lambda} \omega \wedge (-i)^{p-q} (-1)^{\frac{(p +q)(p+q +1) }{2} } \frac{1}{ ( d - p -q )! } L^{ d - p - q } \overline{\omega}  \\
   = & i^{p-q} (-1)^{\frac{(p +q)(p+q -1) }{2} } \frac{1}{ ( d - p -q )! }
     \int_{\Lambda} L^{ d - p - q } \omega \wedge   \overline{\omega}.
  \end{align*}
\end{proof}

  The operator $L$ induces an $\R$-linear map
        $$ L \colon  H_{\Trop,\Dol,\pqs,\t''_{\min}}^{p,q} ( \Lambda ) \ni \omega \mapsto \omega_{Kah} \wedge \omega  \in H_{\Trop,\Dol,\pqs,\t''_{\min}}^{p + 1, q + 1 } ( \Lambda )  . $$
When $\Lambda$ is $\Q$-smooth in codimension $1$, 
  we have $\H_{\min}^{p,q}( \Lambda) = \H_D^{p,q}( \Lambda)$ (\cref{boundary conditions in Q-smooth in codimension 1 case}).

\begin{cor}\label{the Kahler package for tropical Dolbeault cohomology}
  We assume that 
    $\Lambda$ is $\Q$-smooth in codimension $1$,
     and  
    \cref{assumption regularity of Derichlet potentials} 
        hold.

  Then 
       for $p$ and $q$ with $p +q \leq d$,
   \begin{itemize}
      \item (hard Lefschetz theorem)
       we have an isomorphism
            $$ L^{d-p-q} \colon H_{\Trop,\Dol,\pqs,\t''_{\min}}^{p,q} ( \Lambda ) 
              \cong  
              H_{\Trop,\Dol,\pqs,\t''_{\min}}^{d-q,d-p} ( \Lambda ) , $$ 
      and 
      \item (the Hodge-Riemann bilinear relations)
        a symmetric $\R$-bilinear map
        $$H_{\Trop,\Dol,\pqs,\t''_{\min}}^{p,q} ( \Lambda )  \times H_{\Trop,\Dol,\pqs,\t''_{\min}}^{p,q} ( \Lambda ) \ni (\omega, \eta) 
        \mapsto i^{p-q} (-1)^{\frac{(p +q)(p+q -1) }{2} }  \int_{\Lambda} L^{ d - p - q } \omega \wedge \overline{\eta} \in \R$$
            is positive definite on the kernel of $L^{n-p-q +1}$.
    \end{itemize}
\end{cor}
\begin{proof}
  This follows from \cref{the Hodge isomorphism}, \cref{hard Lefschetz}, and \cref{Hodge-Riemann bilinear relation}.
\end{proof}

\section{Density}
In this section, we shall prove \cref{density H = W with boundary conditions}. 
This is an analog of the following fact (\cite[Theorem 8.9 (a)]{Wlo87}):
for an open bouded subset $\Omega \subset \R^n$ with smooth boundary $\partial \Omega$,
the kernel $\ker T$ of the trace operator 
$$T \colon W^{1,2}(\Omega) \to L^2 (\partial \Omega)$$
is the closure of 
$\ker T \cap C^{\infty}(\Omega)$.

\begin{prp}[\cref{density H = W with boundary conditions}]\label{density H = W with boundary conditions2}
  For $\epsilon, \epsilon_1, \epsilon_2 \in \{ \t''_{\min} , \t''_{\max} , \n''_{\min} , \n''_{\max} \} $, 
  the subsets
  $$ \A_{\Trop, \pqs, \epsilon}^{p,q}(\Lambda) \subset W_{\Trop, \epsilon}^{p,q,1,2}(\Lambda)   ,
   \quad \A_{\Trop, \pqs, \epsilon_1,\epsilon_2}^{p,q}(\Lambda) \subset W_{\Trop, \epsilon_1,\epsilon_2}^{p,q,1,2}(\Lambda))$$
  are dense.
\end{prp}

Our strategy of proof is as follows.
For a given $ \omega = (\omega_P)_{P \in \Lambda_{\max}} \in W_{\Trop}^{p,q,1,2}(\Lambda)  $, we can approximate each $\omega_P$ by a smooth differential form $f_P$. 
The problem is to construct $(f_P)_{P \in \Lambda_{\max}}$ which satisfies the same boundary condtion(s) as $\omega$. 
To solve this problem, we shall use \emph{capacity theory}.
By using it, we may assume that the support $\supp (w)$ does not intersects with $(- \log \lvert \cdot \rvert^{-1}) (R)$ for polyhedra $R \in \Lambda$ of dimension at most $d-2$.
Then by partition of unity, we may assume that $\supp (w)$ intersects with $ (- \log \lvert \cdot \rvert )^{-1} (Q)$ for at most one $Q \in \Lambda_{\max -1}^{\ess}$.
Then we can easily construct required approximations $f_P$ by using ``inverses'' of the trace operators (\cite[Theorem 8.8]{Wlo87}, see \cref{review trace operators}) to approximations of traces $\omega_P|_{\partial \exp (-P)_\C \cap  (- \log \lvert \cdot \rvert )^{-1} (Q)} $.

In the following, 
we use the Euclid metric on $\R^r$ and $\C^r$ to define the Sobolev spaces on them.

For a set $E \subset \R^r$, \emph{the Sobolev $2$-capacity} (e.g., \cite[Definition 5.1 and Lemma 5.2]{KLV21}) is 
$\inf_{u \in \mathcal{A'}_2(E)} \lVert u   \rVert^{2}_{W^{1,2}(\R^r)}$, where $\mathcal{A'}_2(E)$ is the subset of $W^{1,2}(\R^r)$ consisting of elements $ 0 \leq u \leq 1$ with $ u  = 1$ almost everywhere in a neighborhood of $E$.
 
  We use the following facts freely. 
\begin{rem}
  \begin{itemize}
    \item The Sobolev $2$-capacity 
  of compact subsets of ``codimension $\geq 2$'' (e.g., polyhedra of codimension $\geq 2$, compact subsets of $\C^d \setminus (\C^{\times} )^d \subset \C^d$)
  is $0$ (see e.g., \cite[Lemma 5.6 and Remark 5.8]{KLV21}).
  \item 
  The subset $C_0^{\infty}( (\C^{\times} )^d ) \subset W^{1,2}(\C^d )$ is dense.
  This can be seen as follows.
  Remind that $C_0^{\infty}( \C^d ) \subset W^{1,2}(\C^d )$ 
  is dense (\cite[Corollary 3.1]{Wlo87}).
  By Lipnitz rule,
  given $f \in C_0^{\infty}( \C^d ) $ 
  can be approximated by 
  $(1-u)f$  for some  $$u \in \mathcal{A'}_2(\supp(f) \cap (\C^d \setminus (\C^{\times} )^d)).$$
  Then $(1-u)f$ has a compact support in $ (\C^{\times} )^d$. 
  A regularizer of $(1-u)f$ (in the sense of \cite[Lemma 3.3]{Wlo87}) is in $C_0^{\infty}( (\C^{\times} )^d ) $ and approximates $(1-u)f$.
  \end{itemize}
\end{rem}

\begin{lem}\label{ average of theta function W 12}
  Let $f \in W^{1,2} ( \C^d ) $.

  Then 
   $$f'(z = r e^{i \theta}):= \frac{1}{ ( 2 \pi )^d } \int_{\theta'= (\theta'_1, \dots,\theta'_d) \in [ 0 , 2 \pi ]^d } 
     f( r_1 e^{i \theta_1'}, \dots, r_d e^{i \theta_d'}) d \theta'_1 \wedge \dots \wedge d \theta'_d $$
   (which is a measurable function on $(\C^{\times})^d $)
    extends to a function in $W^{1,2}(\C^d)$, 
   and 
   there exists 
    a constant $C >0$ (independent of $f$) 
   such that 
    $$ \Arrowvert f' \Arrowvert_{W^{1,2}(\C^d)} 
      \leq C  \Arrowvert f \Arrowvert_{W^{1,2}(\C^d)} .$$
  We also have 
   $$ \frac{\partial}{\partial r_i } \int_{ \theta' = (\theta'_1, \dots,\theta'_d) \in[ 0, 2 \pi ]^d } f  (r e^{i \theta' } ) d \theta'_1 \wedge \dots \wedge d \theta'_d
     =  \int_{ \theta' = (\theta'_1, \dots,\theta'_d)\in[ 0, 2 \pi ]^d } \frac{\partial}{\partial r_i }  f (r e^{ i \theta' } ) d \theta'_1 \wedge \dots \wedge d \theta'_d$$
   for a.a. $r = (r_1,\dots,r_d) \in \R_{>0}^d$.
\end{lem}

\begin{proof}
 When $f \in C_0^{\infty}( (\C^{\times} )^d )$,
 the assertions follow from 
  Fubini's theorem, H\"{o}lder's inequality, and 
  equalities
  \begin{align*}
    \frac{ \partial }{ \partial u_i } & 
    = \cos \theta_i \frac{ \partial }{ \partial r_i }  
    + - \frac{ \sin \theta_i }{ r_i } \frac{ \partial }{ \partial \theta_i }  , \\
    \frac{ \partial }{ \partial v_i } & 
    = \sin \theta_i \frac{ \partial }{ \partial r_i }  
    + \frac{ \cos \theta_i }{ r_i } \frac{ \partial }{ \partial \theta_i }  , 
     \qquad   \qquad  (z = u + i v = r e^{i \theta}) 
    \\
    \frac{ \partial }{ \partial r_i } & 
    = \cos \theta_i \frac{ \partial }{ \partial u_i }  
    +  \sin \theta_i \frac{ \partial }{ \partial v_i }  .
  \end{align*}

 In general, 
 let $f_j \in C_0^{ \infty } ( ( \C^{ \times } )^d ) $ 
  converges to $f$ in $W^{1,2}(\C^d)$ as $j \to \infty $.  
 Then 
  $(f_j')_j $  is a Cauchy sequence in $W^{1,2}(\C^d)$. 
  By H\"{o}lder's inequality, the functions $f_j'$ converges to $f'$ in $L^2 (\C^d)$.  
  Hence we have $f' = \lim_{j \to \infty} f_j' $ in $W^{1,2}(\C^d)$.
 The inequality $ \Arrowvert f' \Arrowvert_{W^{1,2}(\C^d)} 
      \leq C  \Arrowvert f \Arrowvert_{W^{1,2}(\C^d)} $
  and 
 the last assertion follow from the case of $f_j$.
\end{proof}

\begin{thm}[Adams-Hedberg {\cite[Theorem 9.9.1]{AH96}}]\label{Theorem of Adams-Hedberg}
 Let $f_1,\dots,f_r \in W^{1,2}(\R^n) $ $(n \geq 2)$ 
 and $K \subset \R^n$ be a compact subset whose Sobolev $2$-capacity is $0$.

 Then 
  for any $\epsilon > 0$ and any neighborhood $V$ of $K$,
  there exists a compactly supported continuous $\R$-valued function $\omega \in C_0 ( V )  \cap W^{1,2} ( \R^n ) $ 
  such that 
   $0 \leq \omega \leq 1$, we have $\omega =1$ on a neighborhood of $K$, 
   we have $\omega f_i \in W^{1,2}( \R^n)$, 
   and 
   $ \Arrowvert \omega f_i \Arrowvert_{ W^{1,2}(\R^n) } < \epsilon $ for $ 1 \leq i \leq r$.
\end{thm}
\begin{proof}
  \cite[Theorem 9.9.1]{AH96} is the case of $r=1$. 
  We briefly recall their strategy.
  When $f=f_1$ is bounded, 
  it is enough to take 
  a small neighborhood $V'$ of $K$ and
  $\omega \in C_0 ( V' )  \cap W^{1,2} ( \R^n ) $ 
  such that 
  $0 \leq \omega \leq 1$, we have $\omega =1$ on a neighborhood of $K$, and $\Arrowvert \omega \Arrowvert_{ W^{1,2}(\R^n) } $ is small.
  Then, by Lipnitz rule,
  $$\Arrowvert \omega f \Arrowvert_{ W^{1,2}(\R^n) }  
  \leq \Arrowvert f \Arrowvert_{ W^{1,2}(V') }  
   + \sup_{x \in \R^n} \lvert f(x) \rvert \cdot \sum_{i=1}^n \Arrowvert \partial_{x_i}\omega  \Arrowvert_{ L^{2}(\R^n) } 
  $$
  is small.
  When $f$ is not bounded, the last term can be large.
  Therefore, in general, for given large $\lambda>0$, 
  they take $\omega_\lambda$ which is $1$ 
(in particular, $\partial_{x_i} \omega=0$)
    not only on a neighborhood of $K$ but also ``where $ \lvert f \rvert $ is large'' 
  so that we can control the last term
  and 
the required $\omega$ is given as a linear sum of $\omega_\lambda$ (for details see [loc.cit.]).
  We can generalize their proof to $r \geq 2$ by taking maximal function of $\omega_\lambda$ for each $f=f_i$. (Such a function is $1$ ``where some $ \lvert f_i \rvert $ is large''.)

  For $r \geq 2$,
  since most parts of proof are the same as \cite[Theorem 9.9.1]{AH96}
  we only give a rough proof. See [loc.cit.] for details. 
  For each large $\lambda >0$ and $f_i$,
   let $\omega_{i,\lambda} \in C_0^{ \infty } ( \R^n ) $ be as in proof of \cite[Theorem 9.9.1]{AH96} (i.e., $\omega_\lambda$ for $f = f_i$).
   Then we have
   \begin{itemize}
    \item $\supp (\omega_{i,\lambda}) \subset V$,
    \item $0 \leq \omega_{i,\lambda} \leq 1$, 
    \item $\omega_{i,\lambda} =1$ on a neighborhood of $K$,  and 
    \item $ \Arrowvert \omega_{i,\lambda} f_i \Arrowvert_{ W^{1,2}(\R^n) } < A $ for some constant $ A = A(f_i) > 0 .$ 
   \end{itemize}
  We put 
  $$\omega_\lambda ( x ) := \max_{1 \leq i \leq r} \omega_{i,\lambda} ( x ) \in C_0 (V) \cap W^{1,2} ( \R^n ) .$$
  (The maximum of finitely many functions in $W^{1,2}(\R^n)$ is again in $W^{1,2} ( \R^n ) $ (\cite[Corollary 2.1.8]{Zie89}).)
  In the same way as proof of $ \Arrowvert \omega_{i,\lambda} f_i \Arrowvert_{ W^{1,2}(\R^n) } < A  $ in \cite[Theorem 9.9.1]{AH96},
  we have 
  $ \Arrowvert \omega_{\lambda} f_i \Arrowvert_{ W^{1,2}(\R^n) } < C  \ ( 1 \leq i \leq r)$ for some constant $ C = C(f_1,\dots,f_r) > 0 $.
  In the same way as \cite[Theorem 9.9.1]{AH96},
  there are finitely many $a_j >0$ and large $\lambda_j >0$ such that $\sum_j a_j =1$
  and $ \Arrowvert (\sum_j a_j \omega_{\lambda_j}) f_i \Arrowvert_{ W^{1,2}(\R^n) } < \epsilon $ for $ 1 \leq i \leq r$.
  Hence $\omega :=\sum_j a_j \omega_{\lambda_j} $ is a required function.
\end{proof}

\begin{cor}\label{theta invariant version of Adams Hedberg Theorem}
  Let $f_i, n, K, \epsilon,V$ as in \cref{Theorem of Adams-Hedberg}.
  We assume that
    $n$ is even (and we identify $\R^n $ and a smooth toric variety $\C^{\frac{n}{2}}$),
    the subset $K $ is of the form $K= (- \log \lvert \cdot \rvert)^{-1} ( - \log \lvert K \rvert ) $, 
   and 
    $\lvert f_i \rvert $ are of the form $\lvert f_i (z) \rvert  = g_i ( - \log \lvert z \rvert )$
    for some function $g_i$ 
    on an open subset $U_i \subset \C^{\frac{n}{2}}$ such that $U_i =  (- \log \lvert \cdot \rvert)^{-1} ( - \log \lvert U_i \rvert )$.

  Then 
  there exists a compactly supported continuous $\R$-valued function $\omega ( - log \lvert z \rvert ) \in C_0 ( V )  \cap W^{1,2} ( \C^{\frac{n}{2}} ) $ 
  such that 
   $0 \leq \omega \leq 1$, we have $\omega =1$ on a neighborhood of $K$, 
   we have $\omega f_i \in W^{1,2}(U_i)$, 
   and $ \Arrowvert \omega f_i \Arrowvert_{ W^{1,2}(U_i) } < \epsilon $ for $ 1 \leq i \leq r$.
\end{cor}

\begin{proof}
  Let $\omega$ be as in \cref{Theorem of Adams-Hedberg} for sufficiently small $\epsilon' > 0$.
  We put 
   $$\omega'(z = r e^{i \theta }):= \frac{1}{ ( 2 \pi )^{ \frac{n}{2} } } \int_{\theta' = (\theta'_1, \dots,\theta'_{\frac{n}{2}})\in [ 0 , 2 \pi ]^{\frac{n}{2}} } 
    \omega(r e^{i \theta'}) d \theta'_1 \wedge \dots \wedge d \theta'_{\frac{n}{2}}  .$$
  We have $\omega' \in W^{1,2}( \R^n )$ by \cref{ average of theta function W 12}.
  It suffices to show that $\omega' f_i \in W^{1,2}(U_i)$ and 
   $ \Arrowvert \omega' f_i \Arrowvert_{ W^{1,2}(U_i) } < \epsilon $ for $ 1 \leq i \leq r$.
   We put $x_j$ a coordinate of $\R^n$.
  By construction, 
  we can take 
   $ \Arrowvert \omega'  f_i \Arrowvert_{ L^{2}(U_i) } $ 
   and 
   $ \Arrowvert \omega' \frac{\partial}{ \partial x_j } f_i \Arrowvert_{ L^{2}(U_i) } $ small enough.
  By the assumption on $\lvert f_i \rvert$ and the last assertion of \cref{ average of theta function W 12} for $\omega$,
  we can take 
   $ \Arrowvert (\frac{\partial}{ \partial x_j } \omega')  f_i \Arrowvert_{ L^{2}(U_i) } $ small enough.
  In particular, we have
  $\omega' f_i \in W^{1,2} ( U_i )$. 
  This $\omega'$ is a required function.
\end{proof}

\begin{lem}\label{approximation by multiplying continuous functions}
  Let $R \in \Lambda^{\ess}$ be a polyhedron, and
  $\sigma \in \Sigma$ the maximal cone 
  among
  $R \cap N_{\sigma,\R} \neq \emptyset$.
  Let  
    $ \alpha = ( \alpha_P )_{P \in \Lambda_{\max}}\in W_{\Trop}^{p,q,1,2}( \Lambda )$
   such that   
    $$\supp ( \alpha ) \subset T_{\sigma}(\C) \cap  \bigcup_{ \substack{P \in \Lambda_{\max}  \\ R \subset P}} 
    \bigg( \exp ( - P)_\C  \setminus  \bigcup_{ \substack{ R' \in \Lambda^{\ess} \\ R \not\subset R', \ R' \subset P }} 
     ( - \log \lvert \cdot \rvert )^{-1} ( R') \bigg) $$
   (i.e., $\supp (\alpha)$ is contained in a (small) open neighborhood of $\relint ( (-\log \lvert \cdot \rvert )^{-1} (R)) $ in $\bigcup_{P \in \Lambda_{\max} } \exp (-P)_\C $).
  We put 
   $$ K:= 
    \begin{cases}
       \supp \alpha  \setminus (\C^{\times} )^n  & ( \dim R \geq d- 1 )  \\
        ( \supp \alpha \setminus (\C^{\times} )^n   ) \cup  \bigcup_{ \substack{P \in \Lambda_{\max} \\ R \subset P}} (\exp ( - P)_\C  \cap (-\log \lvert \cdot \rvert )^{-1} ( R ))  & ( \dim R \leq d - 2 ).
    \end{cases}
   $$

  Then for any $\epsilon >0$, 
  there exists 
   a compactly supported continuous $\R$-valued function $\omega := \omega ( - \log \lvert z \rvert ) $ on $ T_{\sigma}(\C) \cap \bigcup_{P \in \Lambda_{\max}} \exp ( - P)_\C  $ 
  such that  
  \begin{itemize}
    \item $0 \leq \omega \leq 1$, 
    \item $\omega =1$ on a neighborhood of $K$,
    \item $\omega \alpha \in  W_{\Trop}^{p,q,1,2} ( \Lambda ) $, and  $ \Arrowvert \omega \alpha \Arrowvert_{W^{1,2}(\Lambda)}  < \epsilon $.
  \end{itemize}
\end{lem}

\begin{rem}
  When $\alpha $ satisfies one of four boundary condition in \cref{definition of boundary conditions}, 
  the superform $\omega \alpha$ also satisfies the same one.
\end{rem}

\begin{proof}
 We identify 
  \begin{align*}
   T_{\sigma} ( \C ) & \cong  \C^r \times (\C^{\times})^{n-r}
   \qquad  \qquad  ( r := \dim \sigma)
    \\
   \Trop( T_{\sigma}  ) & \cong  (\R \cup \{ \infty \})^r \times \R^{n-r}  .
  \end{align*}
 Let 
  $ \Phi_1 \colon \Z^{n-r} \to \Z^{d-r}$
  be a $\Z$-linear map 
 such that 
  $$ \Phi_\C := \Id_{ \C^r } \times \Phi_1\otimes_\Z \C^{\times} 
  \colon \C^r \times (\C^{\times})^{n-r}  \to \C^r \times (\C^{\times})^{d-r}  $$
  is injective on each $\exp ( -P_\C ) \cap T_{\sigma} ( \C ) \ ( P \in \Lambda_{\max}$). 
 For each $P \in \Lambda_{\max}$ containing $R$, 
 there exists 
  a $d$-dimensional tropically compact polyhedron 
  $P_\alpha \subset P \cap \Trop (T_\sigma)$ 
 such that $\supp ( \alpha_P ) \subset \exp ( - P_{\alpha })_\C $.
 We can write 
  $$\alpha_P = \sum'_{I,J} \alpha_{ P, I, J } d z_{ P , I } \wedge d \overline{z_{ P , J } } $$
 with 
  $ \alpha_{P, I, J } \in W^{1,2} ( \relint ( \exp ( - P_\alpha)_\C ) )  $ 
 such that 
  $$ \lvert \alpha_{P, I, J } (z_P) \rvert = \beta_{P,I, J} ( - \log \lvert z_P \rvert ) $$
  on $\exp ( - P_\alpha)_\C$
 for some function $\beta_{P, I, J}$, 
   where $z_{P,1}, \dots, z_{P,d} \in M$ is a holomorphic coordinate of $\exp ( - P)_\C $.
 There exists 
   a continuous linear extension operator (\cite[Theorem 5.4]{Wlo87})
  $$ F_{P_\alpha} \colon W^{1,2}(\relint ( \Phi_\C( \exp ( - P_\alpha)_\C ) ) )
  \to W^{1,2}( \C^d ). $$
 By applying \cref{theta invariant version of Adams Hedberg Theorem} 
  to $ F_{P_\alpha} ( \Phi_\C ( \alpha_{P, I, J } ) )  $ 
      ($I,J$ and $P \in \Lambda_{\max}$ containing $R$),
      the compact set $ \Phi_\C ( K ) $,
      and 
      open subsets $\relint ( \Phi_\C ( \exp ( - P_\alpha)_\C  ) )$,
 there exists a compactly supported continuous $\R$-valued function $\omega_0 ( - \log \lvert z \rvert )$ on $\C^r \times (\C^{\times})^{d-r}$ 
 such that 
 $ \omega:=  \Phi_\C^{-1}(\omega_0)$ is a required function. 
\end{proof}

\begin{lem}\label{approximation in codimension 1}
 Let $A \subset \R^d$ be a $d$-dimensional compact polyhedron.
  
 Then 
  there exists 
   a constant $C>0$ 
  such that 
   for 
    $\omega \in W_{\Trop}^{1,2} ( \relint ( \exp ( - A)_\C  ) )$
    and 
    $f \in C^{\infty}(\partial  \exp ( - A)_\C   )$ (i.e., the restriction of a smooth function defined on a neighborhood of $\partial \exp ( - A)_\C  $),
   there exists 
    $\tilde{f} (z  )\in C_{\Trop,\qs}^{\infty} ( \exp (- A)_\C)$
   such that 
    \begin{align*}
     & \tilde{f} (z = r e^{i \theta }) |_{ \partial ( \exp ( - A)_\C  ) } 
       = \frac{1}{ (2 \pi )^d } \int_{\theta'= (\theta'_1, \dots,\theta'_d) \in [0,2\pi ]^d } f (  r e^{i \theta' } ) d \theta'_1 \wedge \dots \wedge d \theta'_d \\
     & \Arrowvert \omega - \tilde{f} \Arrowvert_{W^{1,2} ( \relint ( \exp ( - A)_\C  ) ) } 
       < C \Arrowvert \omega |_{ \partial ( \exp ( - A)_\C  ) } - f \Arrowvert_{W^{ \frac{1}{2}, 2 } ( \partial \exp ( - A)_\C   ) } ,
    \end{align*}
where 
$\Arrowvert  \cdot \Arrowvert_{W^{ \frac{1}{2}, 2 } ( \partial \exp ( - A)_\C   )}$
is a fixed fractional Sobolev norm 
(\cite[Definition 3.1 and Definition 4.4]{Wlo87}).
\end{lem}

\begin{proof}
  We put $g$
   the image of $\omega |_{ \partial ( \exp ( - A)_\C  ) } - f $
   under a continuous linear extension operator (\cite[Theorem 8.8]{Wlo87})
    $$Z \colon W^{ \frac{1}{2}, 2 } ( \partial \exp ( - A)_\C   )
        \to W^{ 1, 2 } ( \relint ( \exp ( - A)_\C  ) ) . $$
  Let $\tilde{f}_1 \in C^{\infty}( \exp ( - A)_\C  ) $ be an extension of $f$.
  Let $h \in C_0^{\infty} ( \relint ( \exp ( - A)_\C  ) ) $ be an  approximation of 
  $\omega - \tilde{f}_1 - g  \in W^{1,2} ( \relint ( \exp ( - A)_\C  ) ) $ (\cite[Theorem 8.9 (a)]{Wlo87}).
  We put  
       $$\tilde{f} (z = r e^{i \theta}) := \frac{1}{ (2 \pi )^d } \int_{\theta'= (\theta'_1, \dots,\theta'_d) \in [0,2\pi ]^d } ( \tilde{f}_1 + h )  ( r e^{i \theta' } ) d \theta'_1 \wedge \dots \wedge d \theta'_d  \in C_{\Trop,\qs}^{\infty} ( \exp ( -A)_\C ).$$
  Then we have 
   \begin{align*}
    \Arrowvert \omega - \tilde{f} \Arrowvert_{W^{1,2} ( \relint ( \exp ( - A)_\C  ) ) } 
    & \leq C_1 \Arrowvert \omega - (\tilde{f}_1 + h ) \Arrowvert_{W^{1,2} ( \relint ( \exp ( - A)_\C  ) ) } \\
    & \leq C_1 (\Arrowvert \omega - (\tilde{f}_1 + h +g  ) \Arrowvert_{W^{1,2} ( \relint ( \exp ( - A)_\C  ) ) } 
     + \Arrowvert g \Arrowvert_{W^{1,2} ( \relint ( \exp ( - A)_\C  ) ) } ) \\
    & \leq (C_1 +1) \Arrowvert g \Arrowvert_{W^{1,2} ( \relint ( \exp ( - A)_\C  ) ) } \\
    & \leq C \Arrowvert \omega |_{ \partial \exp ( - A)_\C  } - f \Arrowvert_{W^{\frac{1}{2} , 2 } ( \partial \exp ( - A)_\C  ) } .
   \end{align*}
  (The first inequality follows from \cref{ average of theta function W 12}.)
\end{proof}

\begin{proof}[Proof of \cref{density H = W with boundary conditions2}]
 By 
  partition of unity by $\A_{\Trop}^{0,0} ( \subset \A_{\Trop,\qs}^{0,0})$ (\cite[Lemma 2.7]{JSS19})  and \cref{approximation by multiplying continuous functions}, 
 it suffices to show that 
  every $\omega = ( \omega_P )_{P \in \Lambda_{\max}} \in W_{\Trop}^{p,q,1,2} ( \Lambda )$ 
   satisfying $\epsilon$ (resp. $\epsilon_1$ and $\epsilon_2$) 
   such that 
    $$\supp \omega 
       \subset
       (\C^{\times})^n \cap 
        \bigcup_{ \substack{P \in \Lambda_{\max}  \\ Q \subset P}} 
        \bigg( \exp ( - P)_\C  \setminus  \bigcup_{ \substack{ R \in \Lambda^{\ess} \\ Q \not\subset R, \ R \subset P }} 
        ( - \log \lvert \cdot \rvert )^{-1} ( R) \bigg)
         $$
    for some $Q \in \Lambda_{\max -1}^{\ess}$
  can be approximated by 
   $f \in \A_{\Trop,\pqs}^{p,q} (\Lambda )$ 
    satisfying the same condition $\epsilon$ (resp. $\epsilon_1$ and $\epsilon_2$). 
    (In particular, for $Q' \in \Lambda_{\max}^{\ess} \setminus \{Q\}$, we have $\omega_{P'}|_{(-\log \lvert \cdot \rvert )^{-1}(Q') } = 0$ for any $P' \in \Lambda_{\max}$ containing $Q'$.)
  
  This easily follows from \cref{approximation in codimension 1} as follows.
  We consider the case of single condition $\t''_{\min}$. 
   The other conditions are similar, and we omit them.
  For each $P \in \Lambda_{\max}$ containing $Q$,
  we fix 
   a $d$-dimensional compact polyhedron $A_P \subset P \cap \R^n$
   such that 
   $\supp \omega_P \subset \exp ( - A_P)_\C$
   and 
   $$\supp \omega_P \cap (- \log \lvert \cdot \rvert)^{-1}(Q) \subset  \relint ( \partial \exp ( - A_P)_\C  \cap (- \log \lvert \cdot \rvert)^{-1}(Q) ).$$
  Let 
   $$\eta_Q = 
      \sum'_{I,J} \eta_{Q,I,J} ( - \log \lvert z \rvert ) 
       d ( - \frac{i}{2} \log z_I ) \wedge d ( - \frac{1}{2} \log \overline{z_J})$$
   be as in definition of condition $\t''_{\min}$
   such that 
   $$\supp \eta_{Q,I,J} \cap \partial \exp ( -P_\C ) \subset \relint ( \partial \exp ( - A_P)_\C  \cap (- \log \lvert \cdot \rvert)^{-1}(Q) ) $$
   ($I,J, P \in \Lambda_{\max}$ containing $Q$).
  By a natural homeomorphism 
    \begin{align*}
    & \partial \exp ( - P)_\C   \cap (- \log \lvert \cdot \rvert )^{-1} (Q) \cap (\C^{\times})^n \\
    \cong &  ( Q \cap \R^n ) \times (S^1)^d \\
    \cong &  \partial \exp ( - P')_\C   \cap (- \log \lvert \cdot \rvert )^{-1} (Q) \cap (\C^{\times})^n
    \end{align*}
    for any $P , P' \in \Lambda_{\max}$ containing $Q$,
   we have a homeomorphism 
   \begin{align*}
    \Psi_{P,P'} \colon &\{ \omega  \in W^{\frac{1}{2},2}(\partial \exp ( - P)_\C    )
    \mid \supp \omega \subset \relint(\partial \exp ( - P)_\C   \cap (- \log \lvert \cdot \rvert )^{-1} (Q) ) \} \\
    \cong &
    \{ \omega  \in W^{\frac{1}{2},2}(\partial \exp ( - P')_\C    )
    \mid \supp \omega \subset \relint(\partial \exp ( - P')_\C   \cap (- \log \lvert \cdot \rvert )^{-1} (Q) ) \} 
   \end{align*}
   which send the restriction of $\eta_{I,J}(- \log \lvert z \rvert )$ to the restriction of it.

  In particular, we can take an approximation $f_{P,I,J} \in C^{\infty} ( \partial  \exp ( - A_P)_\C  )$ of 
  the function
  $\eta_{I,J}( - \log \lvert z \rvert )|_{ \partial \exp (-P)_\C}$
   in $W^{\frac{1}{2} ,2} ( \partial \exp ( - A_P)_\C ) $
   with 
    $$\supp f \subset \relint (\partial  \exp ( - A_P)_\C \cap (- \log \lvert \cdot \rvert )^{-1} (Q)  ) $$
    (\cite[Lemma 3.3 and Definition 4.4]{Wlo87})
   such that $f_{P,I,J} \mapsto f_{P',I,J}$ under the homeomophism $\Psi_{P,P'}$.
  Hence by applying \cref{approximation in codimension 1},
  we get a piece-wise quasi-smooth superform which approximates $\omega$ and satisfies $\t''_{\min}$.
\end{proof}


\begin{thebibliography}{99}


\bibitem[AH96]{AH96}
Adams, D. R., Hedberg, L. I., \textit{Function spaces and potential theory},
Grundlehren der mathematischen Wissenschaften [Fundamental Principles of Mathematical Sciences], 314.\ Springer-Verlag, Berlin, 1996. 

\bibitem[AHK18]{AHK18}
Adiprasito, K., Huh, J.,  Katz, E., \textit{Hodge theory for combinatorial geometries},
Ann.\ of Math.\ (2) \textbf{188} (2018), no.\ 2, 381-452. 

\bibitem[Aks19]{Aks19}
Aksnes, E., \textit{Tropical Poincar\'{e} duality spaces}, master thesis

\bibitem[AP20]{AP20}
Amini, O., Piquerez, M., \textit{Hodge theory for tropical varieties},  
arXiv:2007.07826

\bibitem[AP21]{AP21}
Amini, O., Piquerez, M.,
\textit{Homology of tropical fans}, arXiv:2105.01504

\bibitem[BH17]{BH17}
Babaee, F., Huh, J., \textit{A tropical approach to a generalized Hodge conjecture for positive currents},
Duke Math.\ J.\ \textbf{166} (2017), no.\ 14, 2749-2813.

\bibitem[BHMPW22]{BHMPW22}
Braden, T.,  Huh, J.,  Matherne, J. P.,  Proudfoot, N., Wang, B.,
\textit{A semi-small decomposition of the Chow ring of a matroid},
Adv.\ Math.\ \textbf{409} (2022), Part A, Paper No.\ 108646, 49 pp. 

\bibitem[BL92]{BL92}
Br\"{u}ning, J., Lesch, M., \textit{Hilbert complexes},
J.\ Funct.\ Anal.\ \textbf{108} (1992), no.\ 1, 88-132. 

\bibitem[BGGJK21]{BGGJK21}
Burgos Gil, J. I., Gubler, W., Jell, P., K\"{u}nnemann, K.,
\textit{A comparison of positivity in complex and tropical toric geometry},
Math.\ Z.\ \textbf{299} (2021), no.\ 3-4, 1199-1255. 

\bibitem[CLD12]{CLD12}
Chambert-Loir, A., Ducros, A., \textit{Formes diff\'erentielles r\'eelles et courants sur les espaces de Berkovich}, arXiv:1204.6277

\bibitem[CLS11]{CLS11}
Cox, D. A., Little, J. B., Schenck, H. K., \textit{Toric varieties}, Graduate Studies in Mathematics, 124.\ American Mathematical Society, Providence, RI, 2011.

\bibitem[GH78]{GH78}
Griffiths, P., Harris, J., \textit{Principles of algebraic geometry},
Pure and Applied Mathematics.\ Wiley-Interscience [John Wiley and Sons], New York, 1978.

\bibitem[Gub16]{Gub16}
Gubler, W., \textit{Forms and currents on the analytification of an algebraic variety (after Chambert-Loir and Ducros)}, Nonarchimedean and tropical geometry, 1-30,
Simons Symp., Springer, [Cham], 2016.

\bibitem[GJR21-1]{GJR21-1}
Gubler, W., Jell, P., Rabinoff, J., 
\textit{Forms on Berkovich spaces based on harmonic tropicalizations},
arXiv:2111.05741

\bibitem[GJR21-2]{GJR21-2}
Gubler, W., Jell, P., Rabinoff, J., 
\textit{Dolbeault Cohomology of Graphs and Berkovich Curves},
arXiv:2111.05747

\bibitem[GK17]{GK17}
Gubler, W., K\"{u}nnemann, K.,
\textit{A tropical approach to nonarchimedean Arakelov geometry},
Algebra Number Theory \textbf{11} (2017), no.\ 1, 77-180. 

\bibitem[Hor94]{Hor94}
H\"{o}rmander, L,
\textit{The analysis of linear partial differential operators. III},
Grundlehren Math.\ Wiss., 274[Fundamental Principles of Mathematical Sciences]
Springer-Verlag, Berlin, 1994. 

\bibitem[Jel16T]{Jel16T}
Jell, P., \textit{Real-valued differential forms on Berkovich analytic spaces and their cohomology}, thesis

\bibitem[Jel16A]{Jel16A}
Jell, P., \textit{A Poincar\'{e} lemma for real-valued differential forms on Berkovich spaces},
Math.\ Z.\ \textbf{282} (2016), no.\ 3-4, 1149-1167. 

\bibitem[JSS19]{JSS19}
Jell, P., Shaw, K., Smacka, J., \textit{Superforms, tropical cohomology, and Poincar\'e duality},
Adv.\ Geom.\ \textbf{19} (2019), no.\ 1, 101-130.


\bibitem[KLV21]{KLV21}
Kinnunen, J., Lehrb\"{a}ck, J., V\"{a}h\"{a}kangas, A.,
\textit{Maximal function methods for Sobolev spaces},
Mathematical Surveys and Monographs, 257.\ American Mathematical Society, Providence, RI, 2021.

\bibitem[Lag11]{Lag11}
Lagerberg, A., \textit{$L^2$-estimates for the $d$-operator acting on super forms},
arXiv:1109.3983

\bibitem[Lag12]{Lag12}
Lagerberg, A., \textit{Super currents and tropical geometry},
Math.\ Z.\ \textbf{270} (2012), no.\ 3-4, 1011-1050. 

\bibitem[Mih21]{Mih21}
Mihatsch, A., 
\textit{$\delta$-Forms on Lubin-Tate Space},
arXiv:2112.10018 

\bibitem[Sch95]{Sch95}
Schwarz, G., \textit{Hodge decomposition-a method for solving boundary value problems},
Lecture Notes in Mathematics, 1607.\ Springer-Verlag, Berlin, 1995. 


\bibitem[Tay96]{Tay96}
Taylor, M. E.,
\textit{Partial differential equations},
Texts Appl.\ Math., 23
Springer-Verlag, New York, 1996.

\bibitem[Wei58]{Wei58}
Weil, A., 
\textit{Introduction \`{a} l'\'{e}tude des vari\'{e}t\'{e}s k\"{a}hl\'{e}riennes}, 
Publications de l'Institut de Math\'{e}matique de l'Universit\'{e} de Nancago, VI.\ Actualit\'{e}s Sci.\ Ind.\ no.\ 1267 Hermann, Paris 1958. 

\bibitem[Wel08]{Wel08}
Wells, Jr., R. O.,
\textit{Differential analysis on complex manifolds},
Grad.\ Texts in Math., 65
Springer, New York, 2008. 

\bibitem[Wel76]{Wel76}
Welsh, D.,
\textit{Matroid theory},
London Mathematical Society Monographs, 8, Academic Press, London-New York, 1976.

\bibitem[Wlo87]{Wlo87}
Wloka, J., \textit{Partial differential equations},
Cambridge University Press, Cambridge, 1987.

\bibitem[Zie89]{Zie89}
Ziemer, W.P.,
\textit{Weakly differentiable functions},
Sobolev spaces and functions of bounded variation.\ Graduate Texts in Mathematics, 120.\ Springer-Verlag, New York, 1989.


\end{thebibliography}
\end{document}